\documentclass{amsart}

\usepackage{amssymb,amsmath}\usepackage{stmaryrd}
\usepackage{bbm}
\usepackage{graphics,graphicx}

\usepackage[latin1]{inputenc}
\usepackage{subfigure}

\graphicspath{{./Figures/}}

\newtheorem{theo}{Theorem}
\newtheorem{thm}[theo]{Theorem}
\newtheorem{lem}[theo]{Lemma}
\newtheorem{prop}[theo]{Proposition}
\newtheorem{claim}[theo]{Claim}
\newtheorem{Def}[theo]{Definition}
\newtheorem{cor}[theo]{Corollary}

\newenvironment{Remark}{
\refstepcounter{theo}
\smallbreak\noindent%
\textbf{Remark~\thetheo.}}%
{\medbreak}

\newcommand{\cacher}[1]{}
\newcommand{\pqda}{_{d,a}^{(p,q)}}
\newcommand{\pqde}{_{d,e}^{(p,q)}}
\newcommand{\pqdd}{_{d,d}^{(p,q)}}
\newcommand{\rsb}{_{2b}^{(2r,2s)}}
\newcommand{\rrb}{_{2b}^{(2r,2r)}}
\newcommand{\Brsbc}{_{2b,2c}^{(2r,2s)}}
\newcommand{\rsbc}{_{b,c}^{(r,s)}}

\newcommand{\pqd}{_{d}^{(p,q)}}
\newcommand{\pqdp}{_{d,p}^{(p,q)}}
\newcommand{\ppdp}{_{d,p}^{(p,p)}}
\newcommand{\aada}{_{d,a}^{(a,a)}}
\newcommand{\apda}{_{d,a}^{(a,p)}}
\newcommand{\aqda}{_{d,a}^{(a,q)}}

\newcommand{\apd}{_{d}^{(a,p)}}
\newcommand{\aqd}{_{d}^{(a,q)}}
\newcommand{\Bpqd}{_{2d}^{(2p,2q)}}

\def\cAaad{\cA_{d}^{(a,a)}}

\def\cAppd{\cA_{d}^{(p,p)}}

\def\G{G}

\newcommand{\wh}{\widehat}

\newcommand{\wt}[1]{\widetilde{#1}}
\newcommand{\ds}{\displaystyle}

\def\NN{\mathbb{N}}

\newcommand{\al}{\alpha}
\newcommand{\be}{\beta}
\newcommand{\eps}{\epsilon}
\newcommand{\si}{\sigma}

\def\bn{\bar{n}}
\newcommand{\weight}{w}
\newcommand{\mX}{\mathcal{X}}

\newcommand{\mO}{\mathcal{O}}

\newcommand{\wO}{\widetilde{\mO}}

\newcommand{\mB}{\mathcal{B}}
\newcommand{\mC}{\mathcal{C}}

\newcommand{\nb}{$\mathbb{N}$}
\newcommand{\zb}{$\mathbb{Z}$}

\newcommand{\vv}{\textrm{v}}
\newcommand{\ee}{\textrm{e}}
\newcommand{\ff}{\textrm{f}}

\newcommand{\MQ}{Q_M}

\newcommand{\fig}[3]{\begin{figure}[h!]\begin{center}\includegraphics[#1]{#2}\end{center}\caption{#3}\label{fig:#2}\end{figure}}

\def\gf{generating function }

\def\cA{\mathcal{A}}
\def\cB{\mathcal{B}}

\def\cW{\mathcal{W}}

\def\cW{\mathcal{W}}

\def\bbm{$b/(b\!-\!1)$}
\def\ddm{$d/(d\!-\!2)$}

\newcommand{\titre}[1]{\noindent \textbf{#1}}

\author[O. Bernardi and \'E. Fusy]{Olivier Bernardi$^{*}$ \and \'{E}ric Fusy$^{\dagger}$}
\thanks{$^{*}$Department of Mathematics, MIT, Cambridge, USA,
bernardi@math.mit.edu. Supported by NSF grant DMS-1068626, ANR project A3 and European project ExploreMaps.\\
$^{\dagger}$LIX, \'Ecole Polytechnique, Palaiseau, France, fusy@lix.polytechnique.fr.
Supported by the European project
ExploreMaps (ERC StG 208471)}

\title[Bijections for maps with prescribed degrees and girth]{Unified bijections for maps with prescribed degrees and girth}

\begin{document}


\begin{abstract}
This article presents unified bijective constructions for planar maps, with control on the face degrees and on the girth. Recall that the girth is the length of the smallest cycle, so that maps of girth at least $d=1,2,3$ are respectively the general, loopless, and simple maps. For each positive integer $d$, we obtain a bijection for the class of plane maps (maps with one distinguished root-face) of girth $d$ having a root-face of degree $d$. We then obtain more general bijective constructions for annular maps (maps with two distinguished root-faces) of girth at least $d$.

Our bijections associate to each map a decorated plane tree, and non-root faces of degree $k$ of the map correspond to vertices of degree $k$ of the tree. As special cases we recover several known bijections for bipartite maps, loopless triangulations, simple triangulations, simple quadrangulations, etc. Our work unifies and greatly extends these bijective constructions.

In terms of counting, we obtain for each integer~$d$ an expression for the generating function $F_d(x_d,x_{d+1},x_{d+2},\ldots)$ of plane maps of girth~$d$ with root-face of degree $d$, where the variable~$x_k$ counts the non-root faces of degree~$k$. The expression for~$F_1$ was already obtained bijectively by Bouttier, Di~Francesco and Guitter, but for $d\geq 2$ the expression of~$F_d$ is new. We also obtain an expression for the generating function $\G_{p,q}^{(d,e)}(x_d,x_{d+1},\ldots)$ of annular maps with root-faces of degrees $p$ and $q$, such that cycles separating the two root-faces have length at least $e$ while other cycles have length at least~$d$.

Our strategy is to obtain all the bijections as specializations of a single ``master bijection'' introduced by the authors in a previous article. In order to use this approach, we exhibit certain ``canonical orientations'' characterizing maps with prescribed girth constraints.
\end{abstract}

\maketitle

\section{Introduction}
A planar map is a connected graph embedded without edge-crossing in the sphere.
There is a very rich literature on the enumeration of maps, going back to the seminal work of Tutte~\cite{T62b,Tu63} using generating functions. The approach of Tutte applies to many families of maps (triangulations, bipartite maps, 2-connected maps, etc.) but involves some technical calculations (the \emph{quadratic method} or its generalizations~\cite{BJ06a}; see also~\cite{Eyn} for a more analytic approach). For many families of maps, the generating function turns out to be algebraic, and to have a simple expression in terms of the generating function of a family of trees. Enumerative results for maps can alternatively be obtained by a matrix integral approach~\cite{Bre}, an algebraic approach~\cite{Jackson:character-maps}, or a bijective approach~\cite{Schaeffer:these}.

In the bijective approach one typically establishes a bijection between a class of maps and a class of ``decorated'' plane trees (which are easy to count). This usually gives transparent proofs of the enumerative formulas together with algorithmic 
byproducts~\cite{Poulalhon:triang-3connexe+boundary}. Moreover this approach has proved very powerful for studying the metric properties of maps and solving statistical mechanics models on maps~\cite{BS03,BDG:blockedEdges}. There now exist bijections for many different classes of maps \cite{Boutt,BDFG:mobiles,Fusy:transversal,FuPoScL,Schaeffer:these,Sc97}.

In an attempt to unify several bijections the authors have recently defined a ``master bijection'' $\Phi$ for planar maps~\cite{BeFu10}. It was shown that for each integer $d\geq 3$ the master bijection $\Phi$ can be specialized into a bijection for the class of $d$-angulations of \emph{girth $d$} (the girth of a graph is the minimal length of its cycles). 
This approach has the advantage of unifying two known bijections corresponding to the cases $d=3$ \cite[Sec.~2.3.4]{Schaeffer:these} and $d=4$ \cite[Thm.~4.10]{FuPoScL}. More importantly, for $d\geq 5$ it gives new enumerative results which seem difficult to obtain by a non-bijective approach.


In the present article, we again use the ``master bijection strategy'' and obtain a considerable extension of the results in~\cite{BeFu10}. We first deal with \emph{plane maps}, that is, planar maps with a face distinguished as the \emph{root-face}. For each positive integer~$d$ we consider the class of plane maps of girth~$d$ having a root-face of degree~$d$. We present a bijection between this class of maps and a class of plane trees which is easy to enumerate. Moreover it is possible to keep track of the distribution of the degrees of the faces of the map through the bijection.  Consequently we obtain a system of algebraic equations specifying the (multivariate) generating function of plane maps of girth~$d$ having a root-face of degree~$d$, counted according to the number of faces of each degree. The case $d=1$ had previously been obtained by Bouttier, Di Francesco and Guitter~\cite{Boutt}.

Next we consider \emph{annular maps}, that is, plane maps
with a marked inner face. Annular maps have two girth parameters: 
the \emph{separating girth} and the \emph{non-separating girth} defined respectively as the 
 minimum length of cycles separating and not separating the root face from the marked inner face. For each positive integer~$d$, we consider the class of annular maps of non-separating girth at least~$d$ having separating girth equal to the degree of the root-face. We obtain a bijection between this class of maps and a class of plane trees which is easy to enumerate. Again it is possible to keep track of the distribution of the degrees of the faces of the map through the bijection. 
With some additional work we obtain, for arbitrary positive integers 
$d,e,p,q$, a 
system of algebraic equations specifying the multivariate 
generating function of rooted annular maps
of non-separating  girth at least~$d$, separating girth at least $e$, root-face of degree $p$, and marked inner face of degree $q$,  counted according to the number of faces of each degree.

Using the above result, we prove a universal asymptotic behavior for the number of rooted maps of girth at least~$d$ with face-degrees belonging to a \emph{finite} set $\Delta$. Precisely, the number $c_{d,\Delta}(n)$ of such maps with $n$ faces satisfies $c_{d,\Delta}(n)\sim\kappa\,n^{-5/2}\gamma^n$ for certain computable constants $\kappa,\gamma$ depending on~$d$ and $\Delta$. 
This asymptotic behavior was already established by Bender and Canfield~\cite{BeCa94} in the case of bipartite maps without girth constraint (their statement actually holds for any set $\Delta$,
not necessarily finite). We also obtain a (new) closed formula for the number of rooted simple bipartite maps with given number of faces of each degree.


In order to explain our strategy, we must point out that the master bijection $\Phi$ is a mapping between a certain class of \emph{oriented maps} $\wO$ and a class of decorated plane trees. 
Therefore, in order to obtain a bijection for a particular class of maps $\mC$, one can try to define a ``canonical orientation'' for each map in $\mC$ so as to identify the class $\mC$ with a subset $\wO_\mC\subset\wO$ on which the master bijection $\Phi$ specializes nicely. This is the approach we adopt in this paper, and our main ingredient is a proof that certain (generalized) orientations, called \ddm\emph{-orientations}, characterize the maps of girth~$d$. A special case of \ddm-orientations was already used in \cite{BeFu10} to characterize $d$-angulations of girth~$d$. These orientations are also related to combinatorial structures known as \emph{Schnyder woods}~\cite{BeFu:Schnyder, Sc90}. 

\vspace{.2cm}

\noindent {\bf Relation with known bijections.}
The bijective approach to maps was greatly developed by Schaeffer~\cite{Schaeffer:these} after initial constructions by Cori and Vauquelin~\cite{CoriVa}, and Arqu\`es~\cite{Ar86}.
Most of the bijections for maps are between a class of maps and a class of decorated plane trees. These bijections can be divided into two categories: (A) bijections in which the decorated tree is a spanning tree of the map (and the ``decorations'' are part of the edges not in the spanning trees), and (B) bijections in which the decorated plane tree associated to a map $M$ has vertices of two colors black and white corresponding respectively to the faces and vertices of the map (these bicolored trees are called \emph{mobiles} in several articles)\footnote{This classification comes with two subtleties. First, there are two \emph{dual versions} for bijections of type~B: in one version the decorations of the mobiles are some dangling half-edges, while in the dual version the decorations are some labellings of the vertices; see \cite[Sec.~7]{OB:covered-maps}. Second, it sometimes happens that a bijection of type A can be identified with a ``degenerate form'' of a bijection of type B in which all the white vertices of the mobiles are leaves; see Section~\ref{sec:special}.}. The first bijection of type A is Schaeffer's construction for Eulerian maps~\cite{Sc97}. The first bijection of type~B is Schaeffer's construction for quadrangulations~\cite{Schaeffer:these} (which can be seen as a reformulation of~\cite{CoriVa}) later extended by Bouttier, Di Francesco and Guitter~\cite{BDFG:mobiles}. Bijections of both types requires one to first endow the maps with a ``canonical structure" (typically an orientation) characterizing the class of maps: \emph{Schnyder woods} for simple triangulations, \emph{2-orientations} for simple quadrangulations, \emph{Eulerian orientations} for Eulerian maps, etc. For several classes of maps, there exists both a bijection of type A and of type~B. For instance, the bijections~\cite{Sc97} and~\cite{BDFG:mobiles} both allow one to count bipartite maps.

The master bijection $\Phi$ obtained in \cite{BeFu10} can be seen as a meta construction for all the known bijections of type $B$ (for maps without matter). The master bijection is actually a slight extension of a bijection introduced by the first author in~\cite{OB:boisees} and subsequently reformulated in~\cite{OB:covered-maps} (and extended to maps on orientable surfaces). In \cite{OB:covered-maps} it was already shown that the master bijection can be specialized in order to recover the bijection for bipartite maps presented in~\cite[Sec.~2]{BDFG:mobiles}. 

In the present article, our bijection for plane maps of girth and outer face degree
equal to~$d$ 
generalizes several known bijections. 
In the case $d=1$ our bijection (and the derived
generating function expressions) coincides with the one described by Bouttier, Di Francesco and Guitter in~\cite{Boutt}. 
In the case $d=2$ (loopless maps), our bijection generalizes and unifies two bijections obtained by Schaeffer in the dual setting. Indeed the bijection for Eulerian maps described in~\cite{Sc97} coincides via duality with our bijection for $d=2$ applied to bipartite maps, and the bijection for bridgeless cubic maps described in~\cite[Sec.~2.3.4]{Schaeffer:these} (which is also described and extended in \cite{PS03a}) coincides via duality with our bijection for $d=2$ applied to triangulations. 
For all $d\geq 3$, our bijection generalizes the bijection for $d$-angulations of girth~$d$ given in \cite{BeFu10}. This includes the cases $d=3$ and $d=4$ (simple triangulations and simple quadrangulations) previously obtained in \cite[Thm.~4.10]{FuPoScL} and \cite[Sec.~2.3.3]{Schaeffer:these}. 
Lastly, a slight reformulation of our construction allows us to include the case $d=0$, recovering a bijection described in~\cite{BDFG:mobiles} for vertex-pointed maps.

In two articles in preparation~\cite{Bernardi-Fusy:Bijection-hypermaps,Bernardi-Fusy:Bijection-irreducible}, we further generalize 
the results presented here. In~\cite{Bernardi-Fusy:Bijection-hypermaps} we extend the master bijection to hypermaps and count hypermaps
with control on a certain girth parameter (which extends the definition of girth of a map) and on the degrees
of the hyperedges and of the faces. In~\cite{Bernardi-Fusy:Bijection-irreducible}, relying on more involved orientations, we
count so-called \emph{irreducible} maps (and hypermaps), and recover in particular the bijections 
for irreducible triangulations~\cite{Fusy:transversal} and for irreducible
quadrangulations~\cite{FuPoScL}.

\vspace{.2cm}

\noindent {\bf Outline.}
In Section~\ref{section:definitions} we gather useful definitions on maps and orientations. In Section~\ref{section:mobile}, we recall the master bijection introduced in~\cite{BeFu10} between a set $\wO$ of (weighted) oriented maps and a set of (weighted) mobiles. From there, our strategy is to obtain bijections for (plane and annular) maps with control on the girth and face-degrees by specializing the master bijection. As explained above, this requires the definition of some (generalized) orientations characterizing the different classes of maps. 

Section~\ref{sec:bij_girth} deals with the class of plane maps of girth~$d$ with root-face degree~$d$. We define a class of (weighted) orientations, called \ddm-orientations, and show that a plane map $M$ of root-face degree~$d$ has girth~$d$ if and only if it admits a \ddm-orientation. Moreover in this case there is a unique \ddm-orientation such that $M$ endowed with this orientation is in $\wO$.
The class of plane maps of girth~$d$ with root-face degree~$d$ is thus identified with a subset of $\wO$. Moreover, the master bijection $\Phi$ specializes nicely on this subset, so that for each $d\geq 1$ we obtain a bijection between plane maps of girth~$d$ with root-face degree~$d$, and a family of decorated plane trees called~$d$\emph{-branching mobiles} specifiable by certain degree constraints. Through this bijection, each inner face of degree $k$ in the map corresponds to a black vertex of degree $k$ in the associated mobile. Some simplifications occur for the subclass of bipartite maps when $d=2b$ (in particular one can use simpler orientations called \bbm-orientations) and our presentation actually starts with this simpler case.

In Section~\ref{sec:bij_girth_annular}, we extend our bijections to annular maps. More precisely, for any integers $p,q,d$ we obtain a bijection for annular maps with root-faces of degrees $p$ and $q$,  with separating girth $p$ and non-separating girth~$d$. The strategy parallels the one of the previous section.

In Section~\ref{sec:count}, we enumerate the families of mobiles associated to the above mentioned families of plane maps and annular maps. Concerning plane maps, we obtain, for each $d\geq 1$, an explicit system of algebraic equations characterizing the series $F_d(x_d,x_{d+1},x_{d+2},\ldots)$ counting rooted plane maps of girth~$d$ with root-face of degree~$d$, where each variable $x_k$ counts the non-root faces of degree $k$
(as already mentioned, only the case $d=1$ was known so far~\cite{Boutt}).  
Concerning annular maps, we obtain for each quadruple $p,q,d,e$ of positive integers, an expression for the series $\G\pqde(x_d,x_{d+1},\ldots)$ counting rooted annular maps of non-separating girth at least~$d$ and separating girth at least $e$ with root-faces of degrees $p$ and $q$, where the variable $x_k$ marks the number of non-root faces of degree $k$. 
From these expressions we obtain asymptotic enumerative results. 
Additionally we obtain a closed formula for the number of rooted simple bipartite maps with given number of faces of each degree, and give an alternative derivation of the enumerative formula obtained in~\cite{WaLe75} for loopless maps.

In Section~\ref{sec:special}, we take a closer look at the cases $b=1$ and $d=1,2$ of our bijections and explain the relations with bijections described in~\cite{Boutt,Schaeffer:these,Sc97}. We also describe a slight reformulation which allows us to include the further case $d=0$ and explain the relation with \cite{BDFG:mobiles}. 

In Section~\ref{sec:proof}, we prove the missing results about \ddm-orientations and \bbm-orientations for plane maps and annular maps.

\medskip

\section{Maps, biorientations and mobiles}\label{section:definitions}
This section gathers definitions about maps, orientations, and mobiles.\\

\titre{Maps.}
A \emph{planar map} is a connected planar graph embedded (without edge-crossing)
in the oriented sphere and considered up to continuous deformation.
The \emph{faces} are the connected components of the complement of the graph. A \emph{plane tree} is a map without cycles (it has a unique face). The numbers $v$, $e$, $f$ of vertices, edges and faces of a map are related by the \emph{Euler relation}: $v-e+f=2$. Cutting an edge~$e$ at its middle point gives two \emph{half-edges}, each incident to an endpoint of $e$ (they are both incident to the same vertex if $e$ is a loop). A \emph{corner} is the angular section between two consecutive half-edges around a vertex. The \emph{degree} of a vertex or face $x$, denoted $\deg(x)$, is the number of incident corners. A $d$-\emph{angulation} is a map such that every face has degree~$d$. \emph{Triangulations} and \emph{quadrangulations} correspond to the cases $d=3$ and $d=4$ respectively. The \emph{girth} of a graph is the minimum length of its cycles.
Obviously, a map of girth~$d$ does not have faces of degree less than~$d$. Note that a map is \emph{loopless}
if and only if it has girth at least~$2$ and is \emph{simple} (has no loops or multiple edges) if and only if it has girth at least~3.
A graph is \emph{bipartite} if its vertices can be bicolored in such a way that every edge connects two vertices of different colors. Clearly, the girth of a bipartite graph is even. Lastly, it is easy to see that a planar map is bipartite if and only if every face has even degree.

A \emph{plane map} (also called \emph{face-rooted map}) is a planar map with a marked face, called the \emph{root-face}. See Figure~\ref{fig:maps}(a).
We think of a plane map as embedded in the plane with the root-face taken as the (infinite) outer face. A \emph{rooted map} (also called \emph{corner-rooted map}) is a map with a marked corner, called the \emph{root}; in this case the \emph{root-face} and \emph{root-vertex} are the face and vertex incident to the root. The \emph{outer degree} of a plane (or rooted) map is the degree of the root-face. The faces distinct from the root-face are called \emph{inner faces}. The vertices, edges, and corners are called \emph{outer} if they are incident to the root-face and \emph{inner} otherwise. A half-edge is \emph{outer} if it lies on an outer edge, and is
\emph{inner} otherwise. 

\begin{figure}[h!]
\begin{center}
\includegraphics[width=.8\linewidth]{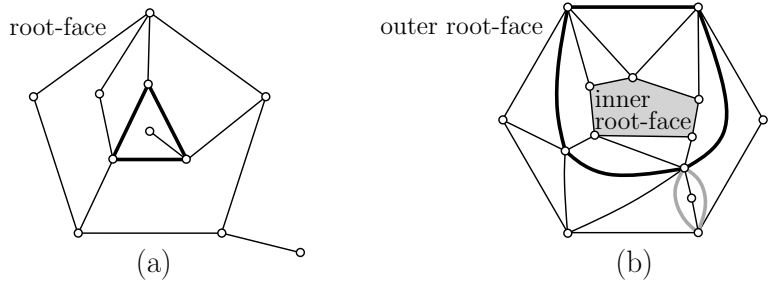}
\end{center}
\caption{(a) A plane map of outer degree $7$ and girth~$3$
(due to the cycle in bold edges). (b) An annular map of outer degree~$6$,
separating girth $4$ (due to the cycle in black bold edges) and non-separating girth~$2$
(due to the cycle in gray bold edges).}
\label{fig:maps}
\end{figure}

An \emph{annular map} is a plane map with a marked inner face.  See Figure~\ref{fig:maps}(b).
Equivalently, it is a planar map with two distinguished \emph{root-faces} called \emph{outer root-face} and \emph{inner root-face} respectively. There are two types of cycles in an annular map: those enclosing the inner root-face are called \emph{separating} and those not enclosing the inner root-face are called \emph{non-separating}. Accordingly, there are two girth parameters: the \emph{separating} (resp. \emph{non-separating}) girth is the minimal length of a separating (resp. non-separating) cycle. We say that an annular map is \emph{rooted} if a corner is marked in each of the root-faces.\\

\titre{Biorientations.}
A \emph{biorientation} of a map $G$, is a choice of an orientation for each half-edge of $G$: each half-edge can be either \emph{ingoing} (oriented toward the vertex), or \emph{outgoing} (oriented toward the middle of the edge). For $i\in\{0,1,2\}$, we call an edge $i$\emph{-way} if it has exactly $i$ ingoing half-edges. Our convention for representing 0-way, 1-way, and 2-way edges is given in Figure~\ref{fig:biorientation}(a). The ordinary notion of \emph{orientation} corresponds to biorientations having only 1-way edges. The \emph{indegree} of a vertex $v$ of $G$ is the number of ingoing half-edges incident to $v$.  Given a biorientation $O$ of a map $G$, a \emph{directed path} of $O$ is a path $P=(v_0,\ldots,v_k)$ of $G$ such that for all $i\in\{0,\ldots, k-1\}$ the edge $\{v_i,v_{i+1}\}$ is either 2-way or 1-way from $v_i$ to $v_{i+1}$. The orientation $O$ is said to be \emph{accessible} from a vertex $v$ if any vertex is reachable from $v$ by a directed path. If $O$ is a biorientation of a plane map, a \emph{clockwise circuit} of $O$ is a simple cycle $C$ of $G$ such that each edge of $C$ is either 2-way or 1-way with the interior of $C$ on its right. A \emph{counterclockwise circuit} is defined similarly. A biorientation of a plane map is said to be \emph{minimal} if it has no counterclockwise circuit. 

\begin{figure}[h!]
\begin{center}
\includegraphics[width=.65\linewidth]{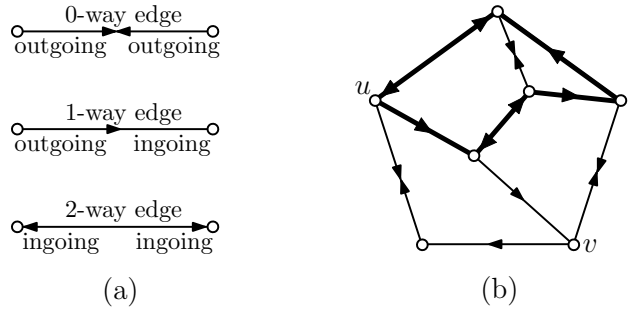}
\end{center}
\caption{(a) Convention for representing 0-way, 1-way and 2-way edges. (b)
A biorientation, which is not minimal (a counterclockwise
circuit is indicated in bold edges). This orientation is accessible from the vertex $u$ but not from the vertex $v$.}
\label{fig:biorientation}
\end{figure}

A biorientation is \emph{weighted} if a \emph{weight} is associated to each half-edge $h$ (in this article the weights will be integers). The \emph{weight} of an edge is the sum of the weights of its half-edges. The \emph{weight} of a vertex $v$ is the sum of the weights of the ingoing half-edges incident to $v$. The \emph{weight} of a face $f$, denoted $\weight(f)$, is the sum of the weights of the outgoing half-edges incident to $f$ and having $f$ on their right; see Figure~\ref{fig:biorientation-weighted}. 
A \emph{\zb-biorientation} is a weighted biorientation where the weight of each half-edge $h$ is an integer which is positive if $h$ is ingoing and non-positive if $h$ is outgoing. An \emph{\nb-biorientation} is a \zb-biorientation where the weights are non-negative (positive for ingoing half-edges, and zero for outgoing half-edges). A weighted biorientation of a plane map is said to be \emph{admissible} if the contour of the outer face is a simple cycle of 1-way edges with weights 0 and 1 on the outgoing and ingoing half-edges, and the inner half-edges incident to the outer vertices are outgoing. 

\begin{Def}\label{def:wO}
A \zb-biorientation of a plane map is said to be \emph{suitable} if it is minimal, admissible, and accessible from every outer vertex (see for instance Figure~\ref{fig:biorientation-weighted}(a)). 
We denote by $\wO$ the set of suitably \zb-bioriented plane maps.
\end{Def}

\begin{figure}[h!]
\begin{center}
\includegraphics[width=.7\linewidth]{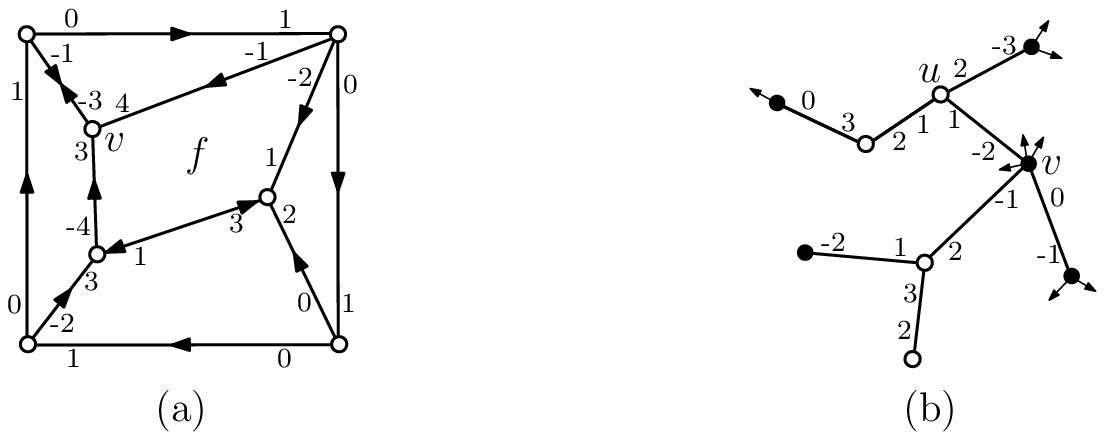}
\end{center}
\caption{(a) A suitably \zb-bioriented plane map. The vertex $v$ has weight $4+3=7$, the face $f$ has weight $-2-4=-6$. (b) A 
\zb-mobile. The white vertex $u$ has weight $1+1+2=4$, the black vertex $v$ has weight $-2-1=-3$ and has degree $6$.}
\label{fig:biorientation-weighted}
\end{figure}

\titre{Mobiles.}
A \emph{mobile} is a plane tree with vertices colored either black or white, and where the black vertices can be incident to some dangling half-edges called \emph{buds}. Buds are represented by outgoing arrows as in Figure~\ref{fig:biorientation-weighted}(b). The \emph{degree} of a black vertex is its number of incident half-edges (including the buds). The \emph{excess} of a mobile is the total number of half-edges incident to the white vertices minus the total number of buds. A \emph{\zb-mobile} is a mobile where each non-bud half-edge $h$ carries a \emph{weight} which is a positive integer if $h$ is incident to a white vertex, and a non-positive integer if $h$ is incident to a black vertex, 
see Figure~\ref{fig:biorientation-weighted}(b). The \emph{weight} of an edge is the sum of the weight of its half-edges. The \emph{weight} of a vertex is the sum of weights of all its incident (non-bud) half-edges.

\bigskip
\section{Master bijection between bioriented maps and mobiles}\label{section:mobile}
In this section we recall the ``master bijection'' $\Phi$ defined in~\cite{BeFu10} (where it is denoted $\Phi_-$) between the set $\wO$ of suitably \zb-bioriented plane maps and a set of \zb-mobiles. The bijection $\Phi$ is illustrated in Figure~\ref{fig:dual_kopening_ex}. It will be specialized in Sections~\ref{sec:bij_girth} and~\ref{sec:bij_girth_annular} to count classes of plane and annular maps.

\begin{Def}\label{def:master-bijections}
Let $M$ be a suitably \zb-bioriented plane map (Definition~\ref{def:wO}) with root-face $f_0$. We view the vertices of $M$ as \emph{white} and place a \emph{black} vertex $b_f$ in each face $f$ of $M$. The embedded graph $\Phi(M)$ with black and white vertices is obtained as follows:
\begin{itemize}
\item Reverse the orientation of all the edges of the root-face (which is a clockwise directed cycle of 1-way edges).
\item For each edge $e$, perform the following operation represented in Figure~\ref{fig:i-way-operation}. Let $h$ and $h'$ be the half-edges of $e$ with respective weights $w$ and $w'$. Let $v$ and $v'$ be respectively the vertices incident to $h$ and $h'$, let $c$, $c'$ be the corners preceding $h$, $h'$ in clockwise order around $v$, $v'$, and let $f$, $f'$ be the faces containing these corners.
\begin{itemize}
\item If $e$ is 0-way, then create an edge between the black vertices $b_f$ and $b_{f'}$ across $e$, and give weight $w$ and $w'$ to the half-edges incident to $b_{f'}$ and $b_f$ respectively. Then, delete the edge $e$.
\item If $e$ is 1-way with $h$ being the ingoing half-edge, then create an edge joining the black vertex $b_f$ to the white vertex $v$ in the corner $c$, and give weight $w$ and $w'$ to the half-edges incident to $v$ and $b_f$ respectively. Then, glue a bud on $b_{f'}$ in the direction of $c'$, and delete the edge $e$.
\item If $e$ is 2-way, then glue buds on $b_f$ and $b_{f'}$ in the directions of the corners $c$ and $c'$ respectively (and leave intact the weighted edge $e$).
\end{itemize}
\item Delete the black vertex $b_{f_0}$, the outer vertices of $M$, and the edges between them (no other edge or bud is incident to these vertices).
\end{itemize}
\end{Def}

\fig{width=11cm}{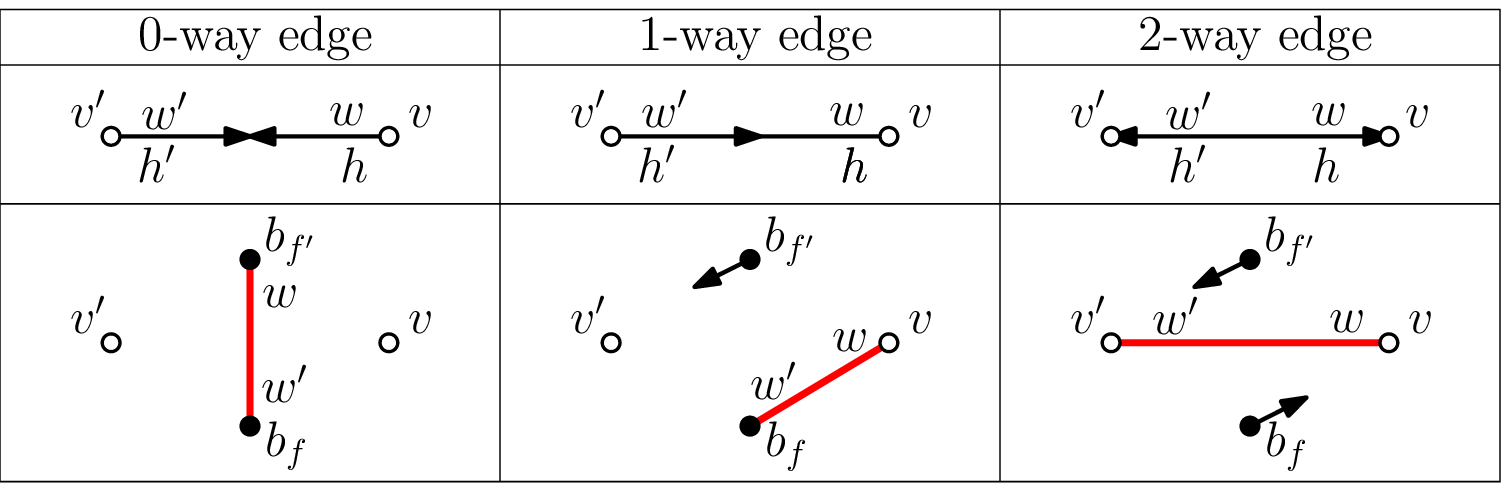}{Local transformation of 0-way, 1-way and 2-way edges done by the bijection $\Phi$.}

\begin{figure}
\begin{center}
\includegraphics[width=\linewidth]{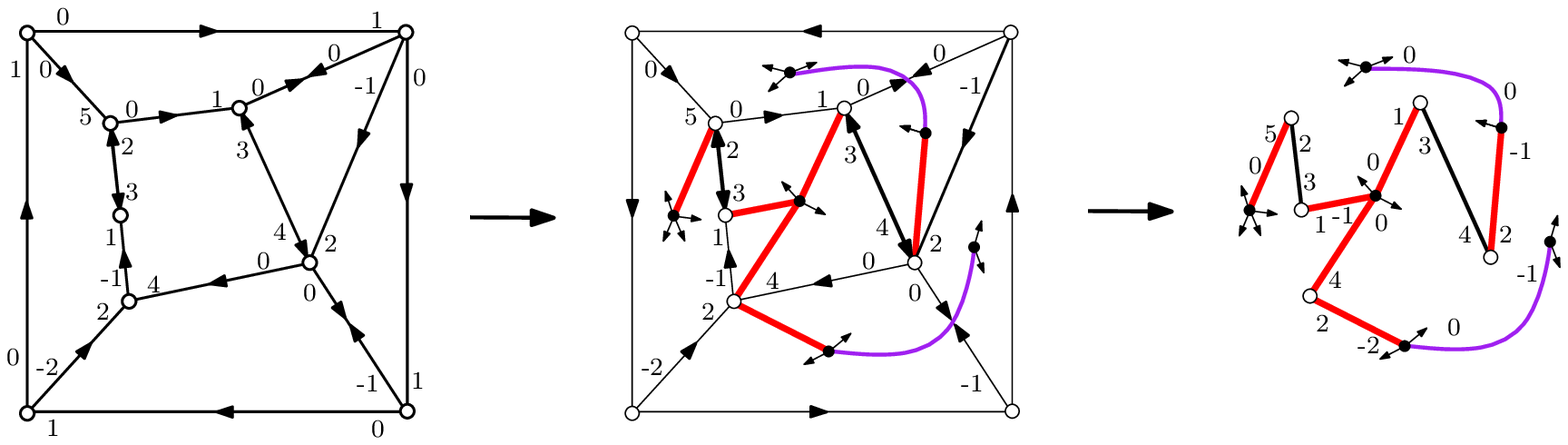}
\end{center}
\caption{The master bijection $\Phi$ applied to a suitably \zb-bioriented plane map.}
\label{fig:dual_kopening_ex}
\end{figure}

The following theorem is proved in~\cite{BeFu10}:

\begin{theo}\label{thm:kmaster-bijections}
The mapping $\Phi$ is a bijection between the set $\wO$ of suitably \zb-bioriented plane maps (Definition~\ref{def:wO}) and the set of \zb-mobiles of negative excess, with the parameter-correspondence given in Figure~\ref{fig:parameter-correspondence}.
\end{theo}

\begin{figure}[h!]
\begin{center}
\includegraphics[width=8cm]{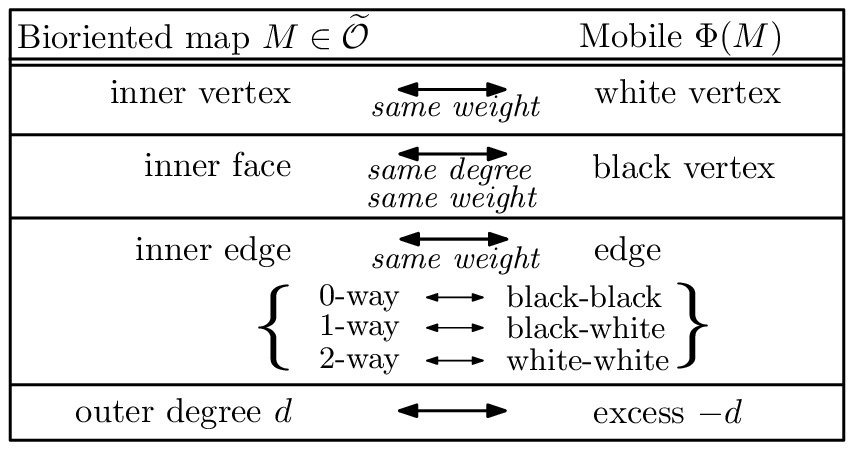}
\end{center}
\caption{Parameter-correspondence of the master bijection $\Phi$.}
\label{fig:parameter-correspondence}
\end{figure}

For $M$ a suitably \zb-bioriented plane map of outer degree~$d$, and $T=\Phi(M)$ the corresponding mobile, we call \emph{exposed} the~$d$ buds of the mobile $T=\Phi(M)$ created by applying the local transformation to the outer edges of $M$ (which have preliminarily been returned). The following additional claim, proved in~\cite{BeFu10}, will be useful for counting purposes.

\begin{claim}\label{claim:bij}
Let $M$ be a suitably \zb-bioriented plane map of outer degree~$d$, and let $T=\Phi(M)$ be the corresponding mobile. There is a bijection between the set $\vec{M}$ of all distinct corner-rooted maps obtained from $M$ by marking an outer corner (note that the cardinality of $\vec{M}$ can be less than~$d$ due to symmetries), 
and the set $\vec{T}$ of all distinct mobiles obtained from $T$ by marking one of the~$d$ exposed buds. Moreover, there is a bijection between the set $T_{\to}$ of mobiles obtained from $T$ by marking a non-exposed bud, and the set $T_{\to\circ}$ of mobiles obtained from $T$ by marking a half-edge incident to a white vertex.
\end{claim}

Before we close this section we recall from~\cite{BeFu10} how to recover the map starting from a mobile (this description will be useful in Section~\ref{sec:special} to compare our bijection with other known bijections). Let $T$ be a mobile (weighted or not) with negative excess $\delta$. The corresponding \emph{fully blossoming} mobile $T'$ is obtained from $T$ by first inserting a \emph{fake} black vertex in the middle of each white-white edge, and then by inserting a dangling half-edge called \emph{stem} in each corner preceding a black-white edge $e$ in clockwise order around the black
extremity of $e$. A fully blossoming mobile is represented in solid lines in Figure~\ref{fig:inverse-master} (buds and stems are respectively indicated by outgoing and ingoing arrows). Turning in counterclockwise direction around the mobile $T'$, one sees a sequence of buds and stems. The \emph{partial closure} of $T'$ is obtained by drawing an edge from each bud to the next available stem in counterclockwise order around $T'$ (these edges can be drawn without crossings). This leaves $|\delta|$ buds unmatched (since the excess $\delta$ is equal to the number of stems minus the number of buds). The \emph{complete closure} $\Psi(T)$ of $T$ is the vertex-rooted bioriented map obtained from the partial closure by first creating a \emph{root-vertex} $v_0$ in the face containing the unmatched buds and joining it to all the unmatched buds, and then deleting all the white-white and black-white edges of the mobile $T$ and erasing the fake black vertices (these were at the middle of some edges); see Figure~\ref{fig:inverse-master}. 

\fig{width=\linewidth}{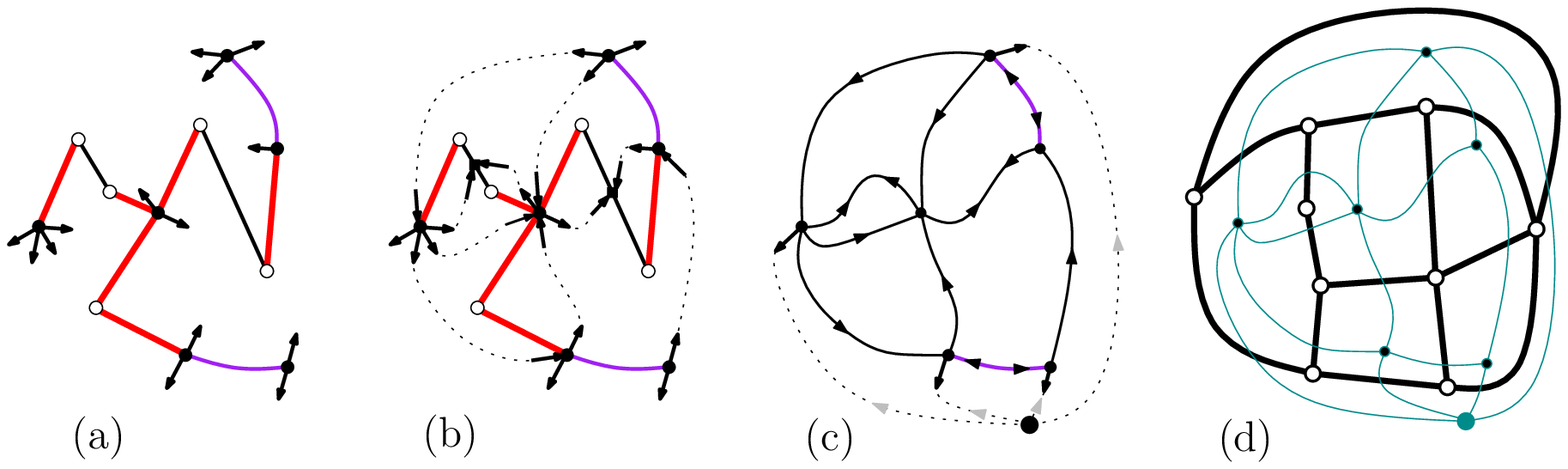}{The mapping $\Psi$. (a) A mobile $T$. (b) The fully blossoming mobile $T'$ (drawn in solid lines with buds represented as outgoing arrows, and stems represented as ingoing arrows) and its partial closure (drawn in dashed lines). (c) The complete closure $\Psi(T)$. (d) The dual of $\Psi(T)$.}

\begin{prop}[\cite{BeFu10}]\label{prop:closure}
Let $M$ be suitably \zb-bioriented plane map, let $T=\Phi(M)$ be the associated mobile, and let $N=\Psi(T)$ be its complete closure.   Then the plane map underlying $M$ is dual to the vertex-rooted map underlying $N$. 
\end{prop}




\section{Bijections for maps with one root-face}\label{sec:bij_girth}
In this section, we present our bijections for plane maps.
For each positive integer~$d$, we consider the class $\mC_d$ of plane maps of outer degree~$d$ and girth~$d$. We define some \zb-biorientations that characterize the maps in $\mC_d$. This allows us to identify the class $\mC_d$ with a subset of $\wO$. We then specialize the master bijection to this subset of $\wO$ and obtain a bijection for maps in $\mC_d$ with control on the number
of inner faces in each degree $i\geq d$. For the sake of clarity we start with the
bipartite case, where the orientations and bijections are simpler.

\subsection{Bipartite case}\label{sec:bip}
In this section, $b$ is a fixed positive integer. 
We start with the definition of the \zb-biorientations that characterize the bipartite maps in~$\mC_{2b}$.
\begin{Def} \label{def:bbm-orientation}
Let $M$ be a bipartite plane map of outer degree $2b$ having no face of degree less than $2b$. A \emph{\bbm-orientation} of $M$ is an admissible \zb-biorientation such that every outgoing half-edge has weight 0 or -1 and
\begin{enumerate}
\item[(i)] each inner edge has weight $b-1$,
\item[(ii)] each inner vertex has weight $b$,
\item[(iii)] each inner face $f$ has degree and weight satisfying $\deg(f)/2+\weight(f)=b$.
\end{enumerate}
\end{Def}

Figure~\ref{fig:bij_bip} shows some \bbm-orientations for $b=2$ and $b=3$. Observe that for $b\geq 2$, a \bbm-orientation has no 0-way edges, while for $b\leq 2$ it has no 2-way edges. 
Definition~\ref{def:bbm-orientation} of \bbm-orientations actually generalizes the one given in~\cite{BeFu10} for $2b$-angulations. Note that \bbm-orientations of $2b$-angulations are in fact \nb-biorientations since Condition~(iii) implies that the weight of every outgoing half-edge is 0.

\begin{thm}\label{thm:bip}
Let $M$ be a bipartite plane map of outer degree $2b$ having no face of degree less than $2b$. Then $M$ admits a \bbm-orientation if and only if $M$ has girth $2b$. In this case, there exists a unique suitable \bbm-orientation of $M$.
\end{thm}

The proof of Theorem~\ref{thm:bip} (which extends a result given in~\cite{BeFu10} for $2b$-angulations) is delayed to Section~\ref{sec:proof}. We now define the class of \zb-mobiles that we will show to be in bijection with bipartite maps in $\mC_{2b}$.
\begin{Def}
A \emph{$b$-dibranching mobile} is a \zb-mobile such that half-edges incident to black vertices have weight 0 or $-1$ and
\begin{enumerate}
\item[(i)] each edge has weight $b-1$,
\item[(ii)] each white vertex has weight $b$,
\item[(iii)] each black vertex $v$ has degree and weight satisfying $\deg(v)/2+\weight(v)=b$; equivalently a black vertex of degree $2i$ is adjacent to $i-b$ white leaves. 

\end{enumerate}
\end{Def}
 
The two ways of phrasing Condition (iii) are equivalent because a half-edge incident to a black vertex has weight -1 if and only if it belongs to an edge incident to a white leaf. Examples of $b$-dibranching mobiles are given in Figure~\ref{fig:bij_bip}. The possible edges of a $b$-dibranching mobile are represented for different values of $b$ in Figure~\ref{fig:b-dibranching-edges}. 

\fig{width=.9\linewidth}{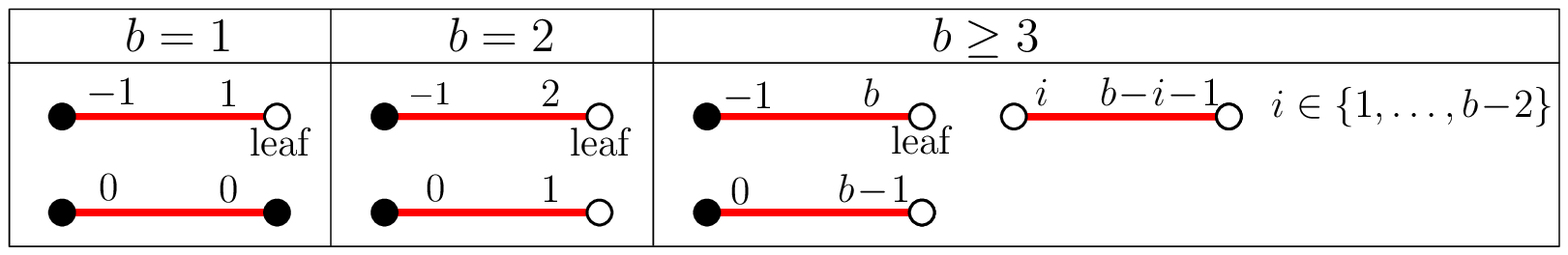}{The possible edges of $b$-dibranching mobiles. The white leaves are indicated.}

\begin{figure}
\begin{center}
\includegraphics[width=\linewidth]{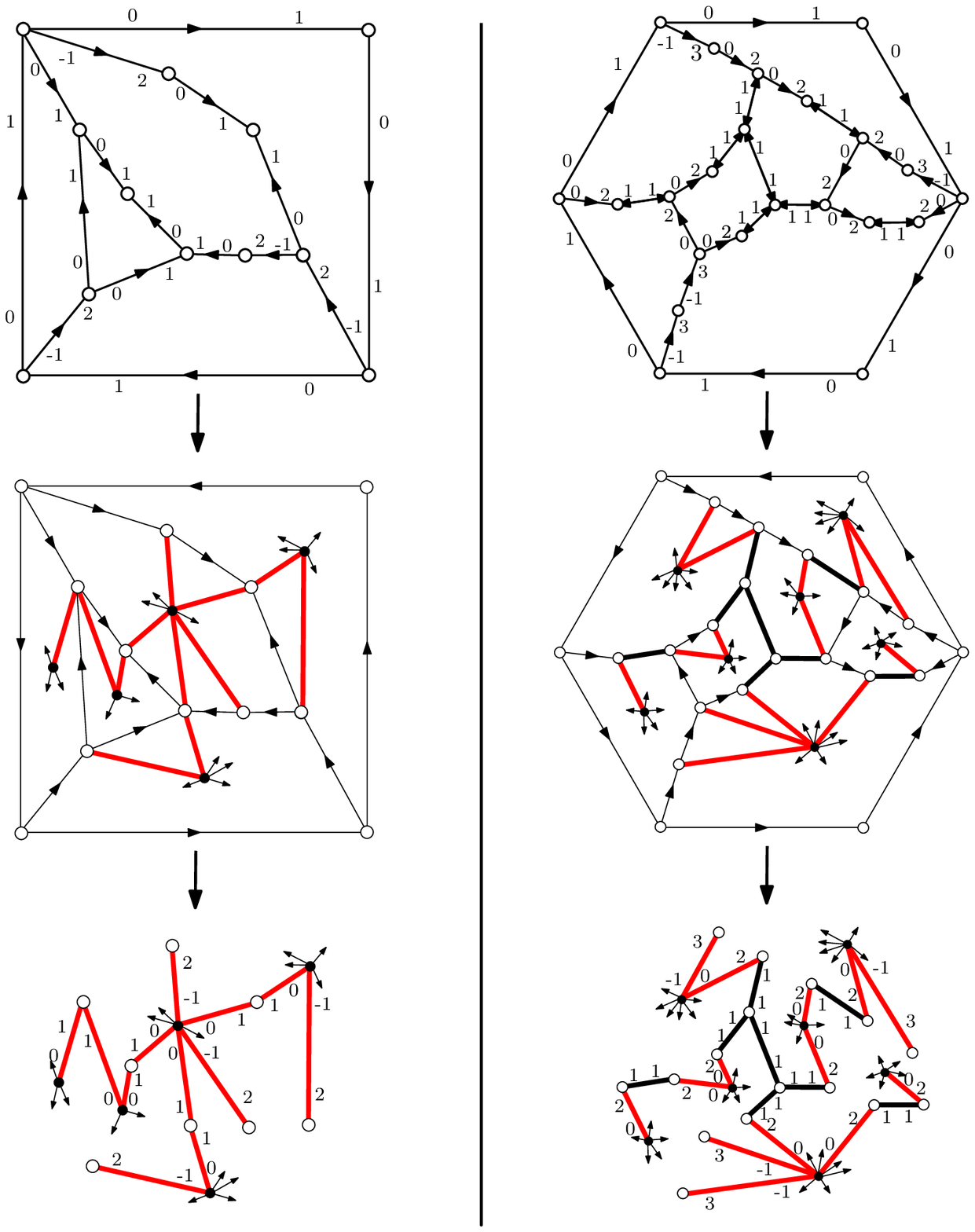}
\end{center}
\caption{Bijections for bipartite maps in the cases $b=2$ (left) and $b=3$ (right).
Top: a plane bipartite map of girth $2b$ and outer degree $2b$ endowed with its suitable \bbm-orientation.
Bottom: the associated $b$-dibranching mobiles.}
\label{fig:bij_bip}
\end{figure}

\begin{claim}\label{claim:excess_bip}
Any $b$-dibranching mobile has excess $-2b$.
\end{claim}
\begin{proof}
Let $T$ be a $b$-dibranching mobile. Let $\ee$ be the number of edges and $\beta$ be the number of buds. Let $v_b$ and $v_w$ be the number of black and white vertices respectively. Let $h_b$ and $h_w$ be the number of non-bud half-edges incident to black and white vertices, respectively. By definition, the excess $\delta$ of the mobile is $\delta=h_w-\beta$. Now, by Condition (iii) on black vertices, one gets $(h_b+\beta)/2+S=b\,v_b$, where $S$ is the sum of weights of the half-edges incident to black vertices. By Conditions (i) and (ii), one gets $\ee(b-1)=b\,v_w+S$. Eliminating $S$ between these relations gives $2\ee(b-1)+h_b+\beta=2b(v_b+v_w)$. Lastly, plugging $v_b+v_w=\ee+1$ and $2\ee=h_b+h_w$ in this relation, one obtains $h_w-\beta=-2b$.
\end{proof}

We now come to the main result of this subsection, which is the correspondence between the set $\mC_{2b}$ of bipartite maps and the $b$-dibranching mobiles. First of all,  Theorem~\ref{thm:bip} allows one to identify the set $\mC_{2b}$ of bipartite maps with the set of \bbm-oriented plane maps in $\wO$. We now consider the image of this subset of  $\wO$ by the master bijection $\Phi$.
In view of the parameter-correspondence induced by the master bijection $\Phi$ (Theorem~\ref{thm:kmaster-bijections}), it is clear that Conditions (i), (ii), (iii) of the \bbm-orientations correspond respectively to Conditions (i), (ii), (iii) of the $b$-dibranching mobiles. Thus, by Theorem~\ref{thm:kmaster-bijections}, the master bijection $\Phi$ induces a bijection between the set of \bbm-oriented plane maps in $\wO$ and the set of $b$-dibranching mobiles of excess $-2b$. Moreover, by Claim~\ref{claim:excess_bip} the constraint on the excess is redundant. We conclude:

\begin{theo}\label{thm:girth2b}
For any positive integer $b$, bipartite plane maps of girth $2b$ and outer degree $2b$ are in bijection
with $b$-dibranching mobiles. Moreover, each inner face of degree $2i$ in the map corresponds
to a black vertex of degree $2i$ in the mobile.
\end{theo}
Figure~\ref{fig:bij_bip} illustrates the bijection on two examples ($b=2$, $b=3$). The case $b=1$ and its relation with~\cite{Sc97} is examined in more details in Section~\ref{sec:special}.

\subsection{General case}\label{sec:gen}
We now treat the case of general (not necessarily bipartite) maps. In this subsection, $d$ is a fixed positive integer. 

\begin{Def} \label{def:ddmorient}
Let $M$ be a plane map of outer degree~$d$ having no face of degree less than~$d$. A \emph{\ddm-orientation} of $M$ is an admissible \zb-biorientation such that every outgoing half-edge has weight 0, $-1$ or $-2$ and
\begin{enumerate}
\item[(i)] each inner edge has weight $d-2$,
\item[(ii)] each inner vertex has weight~$d$,
\item[(iii)] each inner face $f$ has degree and weight satisfying $\deg(f)+\weight(f)=d$.
\end{enumerate}
\end{Def}

Figure~\ref{fig:bij_bip} shows some \ddm-orientations for $d=3$ and $d=5$. The cases $d=1$ and $d=2$ are represented in Figures~\ref{fig:bij_d=1} and~\ref{fig:loopless2} respectively.
Definition~\ref{def:ddmorient} of \ddm-orientations actually generalizes the one given in~\cite{BeFu10} for $d$-angulations. Note that \ddm-orientations of $d$-angulations are in fact \nb-biorientations since Condition~(iii) implies that the weight of every outgoing half-edge is 0.

\begin{thm}\label{thm:gen}
Let $M$ be a plane map of outer degree~$d$ having no face of degree less than~$d$. Then, $M$ admits a \ddm-orientation if and only if $M$ has girth~$d$. In this case, there exists a unique suitable \ddm-orientation of $M$.
\end{thm}

\begin{Remark}\label{rk:bip-specialization} If $d=2b$ and $M$ is a bipartite plane map of outer degree~$d$ and girth~$d$, then the unique suitable \ddm-orientation of $M$ is obtained from its suitable \bbm-orientation by doubling the weight of every inner half-edge (since the \zb-biorientation obtained in this way is clearly a suitable \ddm-orientation).
\end{Remark}

The proof of Theorem~\ref{thm:gen} 
 is delayed to Section~\ref{sec:proof}. We now define the class of mobiles that we will show to be in bijection with $\mathcal{C}_d$. 

\begin{Def}\label{def:dbranching}
For a positive integer~$d$, a \emph{$d$-branching mobile} is a \zb-mobile such that half-edges incident to black vertices have weight 0, $-1$ or $-2$ and
\begin{enumerate}
\item[(i)] each edge has weight $d-2$,
\item[(ii)] each white vertex has weight~$d$,
\item[(iii)] each black vertex $v$ has degree and weight satisfying $\deg(v)+\weight(v)=d$.
\end{enumerate}
\end{Def}

\fig{width=.9\linewidth}{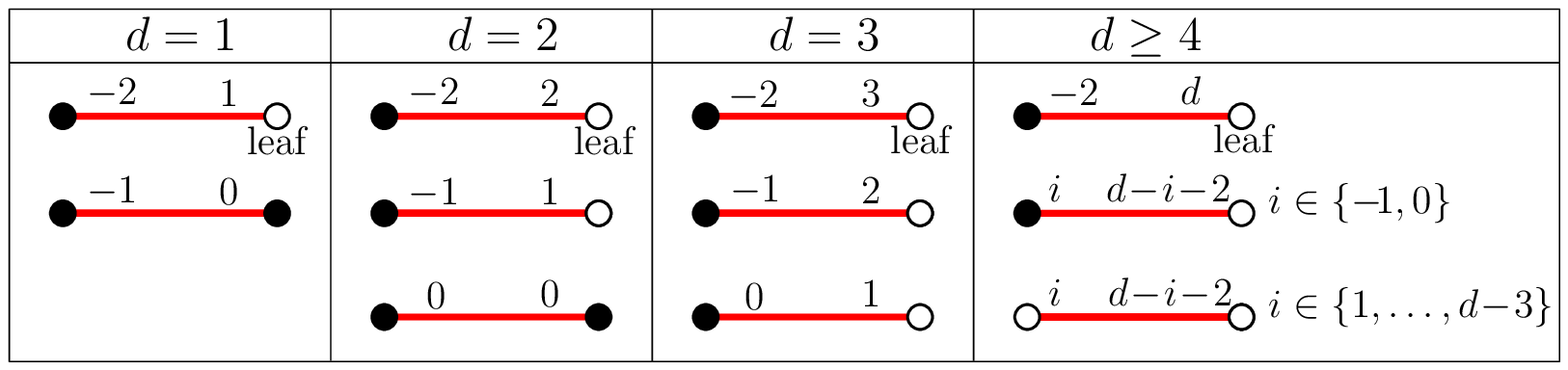}{The possible edges of $d$-branching mobiles. The white leaves are indicated.}

\begin{figure}
\begin{center}
\includegraphics[width=\linewidth]{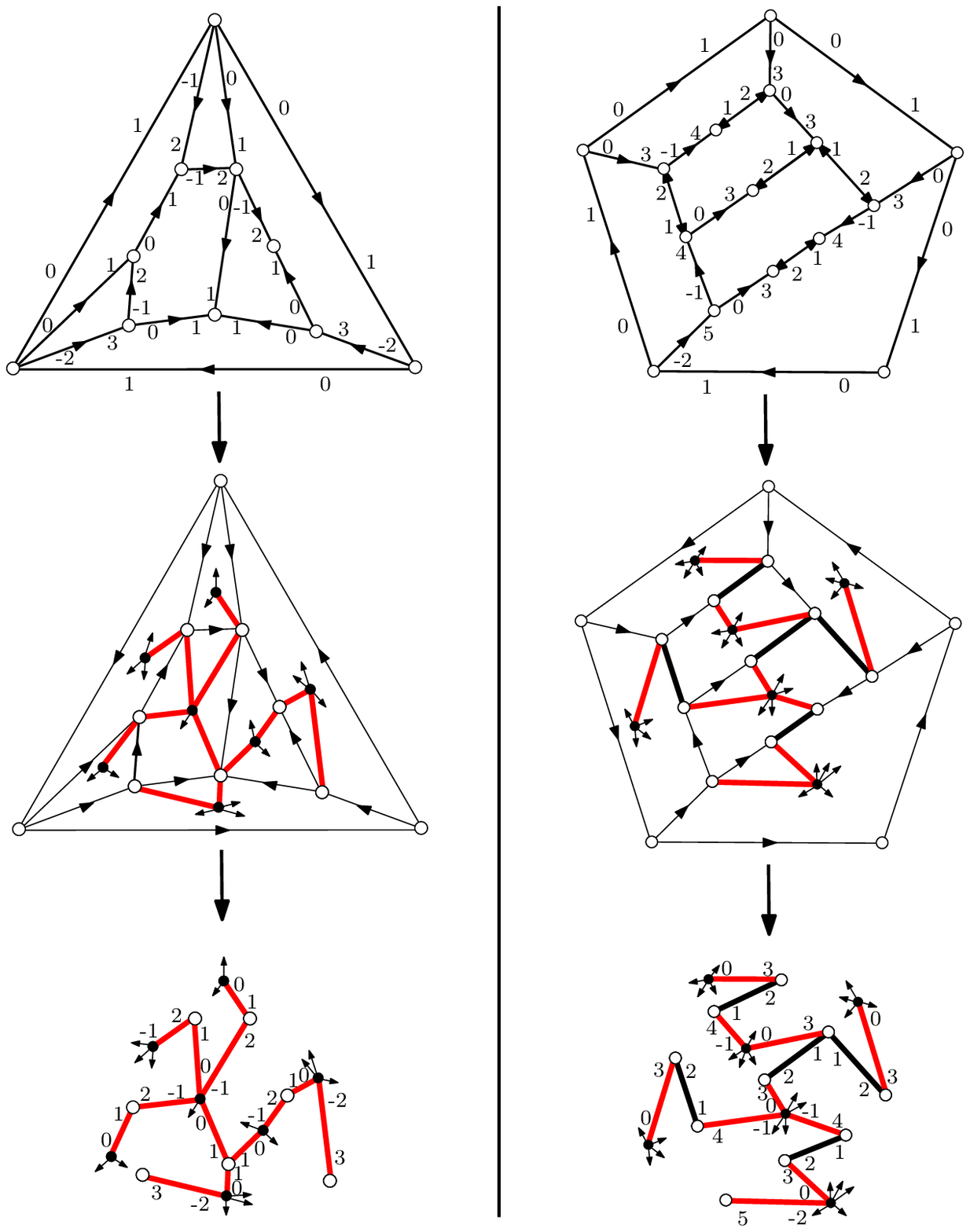}
\end{center}
\caption{
The bijection applied to maps in $\mC_d$ ($d=3$ on the left and $d=5$ on the right). 
Top: a plane map of girth~$d$ and outer degree~$d$ endowed with its suitable \ddm-orientation. 
Bottom: the associated $d$-branching mobiles.}
\label{fig:bij_gen}
\end{figure}

The possible edges of a $d$-branching mobile are represented for different values of~$d$ in Figure~\ref{fig:d-branching-edges}. The following claim can be proved by an argument similar to the one used in Claim~\ref{claim:excess_bip}.

\begin{claim}\label{claim:excess_gen}
Any $d$-branching mobile has excess $-d$.
\end{claim}

We now come to the main result of this subsection, which is the correspondence between the set $\mC_{d}$ of plane maps of girth~$d$ and outer degree~$d$ and the set of $d$-branching mobiles. By Theorem~\ref{thm:gen}, the set $\mC_{d}$ can be identified with the subset of \ddm-oriented plane maps in $\wO$. Moreover, as in the bipartite case, it is easy to see from Theorem~\ref{thm:kmaster-bijections} that the master bijection $\Phi$ induces a bijection between the set of \ddm-oriented plane maps in $\wO$ and the set of $d$-branching mobiles. We conclude:

\begin{theo}\label{thm:girthd}
For any positive integer~$d$, plane maps of girth~$d$ and outer degree~$d$ are in bijection
with $d$-branching mobiles. Moreover, each inner face of degree $i$ in the map corresponds
to a black vertex of degree $i$ in the mobile.
\end{theo}

Figure~\ref{fig:bij_gen} illustrates the bijection on two examples ($d=3$, $d=5$). The bijection of Theorem~\ref{thm:girthd} is actually a generalization of the bijection given in~\cite{BeFu10} for $d$-angulations of girth $d\geq 3$ (for $d$-angulations there are no negative weights). The cases $d=1$ and $d=2$ of Theorem~\ref{thm:girthd} are examined in more details in Section~\ref{sec:special}, in particular the relation between our bijection in the case $d=1$ and the bijection described by Bouttier, Di Francesco and Guitter in~\cite{Boutt} (we also show a link with another bijection described by the same authors in~\cite{BDFG:mobiles}).\\

\begin{Remark} For $d=2b$ it is clear from Remark~\ref{rk:bip-specialization} that the bijection of Theorem~\ref{thm:girth2b} is equal to the specialization of the bijection of Theorem~\ref{thm:girthd} to bipartite maps, up to dividing the weights of the mobiles 
by two. 
\end{Remark}


\section{Bijections for maps with two root-faces}\label{sec:bij_girth_annular}
In this section we describe bijections for annular maps.
An annular map is of \emph{type} $(p,q)$ if the outer and inner root-faces have degrees $p$ and $q$ respectively. 
We denote by $\cA\pqd$ the class of annular maps of type $(p,q)$ with non-separating girth at least~$d$ and separating girth $p$ (in particular $\cA\pqd=\emptyset$ unless $q\geq p$). In the following we obtain a bijection between $\cA\pqd$ and a class of mobiles. 
Our strategy parallels the one of the previous section, and we start again with the bipartite case which is simpler. In Section~\ref{sec:count_annular} we will show that counting results for the classes $\cA\pqd$ can be used to enumerate also the annular maps with separating girth smaller than the outer degree.

\subsection{Bipartite case}\label{sec:bip_annular}
In this subsection we fix positive integers $b,r,s$ with $r\leq s$. We start with the definition of the 
\zb-biorientations that characterize the bipartite maps in $\cA\rsb$.
\begin{Def} \label{def:bbm-orientation_annular}
Let $M$ be a bipartite annular map of type $(2r,2s)$ having no face of degree less than $2b$. A \emph{\bbm-orientation} of $M$ is an admissible \zb-biorientation such that every outgoing half-edge has weight 0 or -1 and
\begin{enumerate}
\item[(i)] each inner edge has weight $b-1$,
\item[(ii)] each inner vertex has weight $b$,
\item[(iii)] each non-root face $f$ has degree and weight satisfying $\deg(f)/2+\weight(f)=b$,
\item[(iv)] the inner root-face has degree $2s$ and weight $r-s$.
\end{enumerate}
\end{Def}

Figure~\ref{fig:bij-annular} (top left) shows a \bbm-orientation for $b=2$.
Note that when $r=b$ (outer root-face of degree $2b$) every inner face (including the inner root-face)
satisfies $\deg(f)/2+\weight(f)=b$, in which
case we recover the definition of \bbm-orientations for plane bipartite maps of outer degree $2b$,
as given in Section~\ref{sec:bij_girth}.


\begin{thm}\label{thm:bip_annular}
Let $M$ be an annular bipartite map of type $(2r,2s)$. Then $M$ admits a \bbm-orientation if and only if $M$ is in $\cA\rsb$.
 In this case, there exists a unique suitable \bbm-orientation of $M$.
\end{thm}

The proof of Theorem~\ref{thm:bip_annular} (which extends Theorem~\ref{thm:bip}, corresponding to the case $r=b$) is delayed to Section~\ref{sec:proof}. We now define the class of \zb-mobiles that we will show to be in bijection with bipartite maps in $\cA\rsb$.
\begin{Def}
A \emph{$b$-dibranching mobile of type $(2r,2s)$} is a \zb-mobile with a marked black vertex called \emph{special vertex} such that half-edges incident to black vertices have weight 0 or $-1$ and
\begin{enumerate}
\item[(i)] each edge has weight $b-1$,
\item[(ii)] each white vertex has weight $b$,
\item[(iii)] each non-special black vertex $v$ has degree and weight satisfying $\deg(v)/2+\weight(v)=b$; 
equivalently a non-special black vertex of degree $2i$ is adjacent to $i-b$ white leaves.
\item[(iv)] the special vertex $v_0$ has degree $2s$ and weight $r-s$; equivalently $v_0$ has degree $2s$ and is adjacent to $s-r$ white leaves.
\end{enumerate}
\end{Def}

A $2$-dibranching mobile of type $(6,8)$ is represented in Figure~\ref{fig:bij-annular} (bottom left). 
As a straightforward extension of Claim~\ref{claim:excess_bip} we obtain:
\begin{claim}\label{claim:excess_bip_annular}
Any $b$-dibranching mobile of type $(2r,2s)$ has excess $-2r$.
\end{claim}

We now come to the main result of this subsection, which is the correspondence between the set $\cA\rsb$ of annular bipartite maps and $b$-dibranching mobiles of type $(2r,2s)$. First of all, by Theorem~\ref{thm:bip_annular} the set $\cA\rsb$ can be identified with the subset of \bbm-oriented annular maps of type $(2r,2s)$ in $\wO$. Thus, it remains to show that the master bijection induces a bijection between this subset and the set of $b$-dibranching mobiles of type $(2r,2s)$. In view of the parameter-correspondence of the master bijection $\Phi$ (Theorem~\ref{thm:kmaster-bijections}), it is clear that Conditions (i), (ii), (iii), (iv) of the \bbm-orientations correspond respectively to Conditions (i), (ii), (iii), (iv) of the $b$-dibranching mobiles. Thus, by Theorem~\ref{thm:kmaster-bijections}, the master bijection $\Phi$ induces a bijection between the set of \bbm-oriented annular maps of type $(2r,2s)$ in $\wO$ and the set of $b$-dibranching mobiles of type $(2r,2s)$ and excess $-2r$. Moreover, by Claim~\ref{claim:excess_bip_annular} the constraint on the excess is redundant. We conclude:

\begin{theo}\label{thm:girth2b_annular}
Bipartite annular maps in $\cA\rsb$ are in bijection with $b$-dibranching mobiles of type $(2r,2s)$. Moreover, each non-root face of degree $2i$ in the map corresponds to a non-special black vertex of degree $2i$ in the mobile.
\end{theo}
Theorem~\ref{thm:girth2b_annular} is illustrated in Figure~\ref{fig:bij-annular} (left). 
Observe that the case $b=r$ in Theorem~\ref{thm:girth2b_annular} corresponds to the bijection of Theorem~\ref{thm:girth2b} 
where an inner face is marked. 

\begin{figure}
\begin{center}
\includegraphics[width=\linewidth]{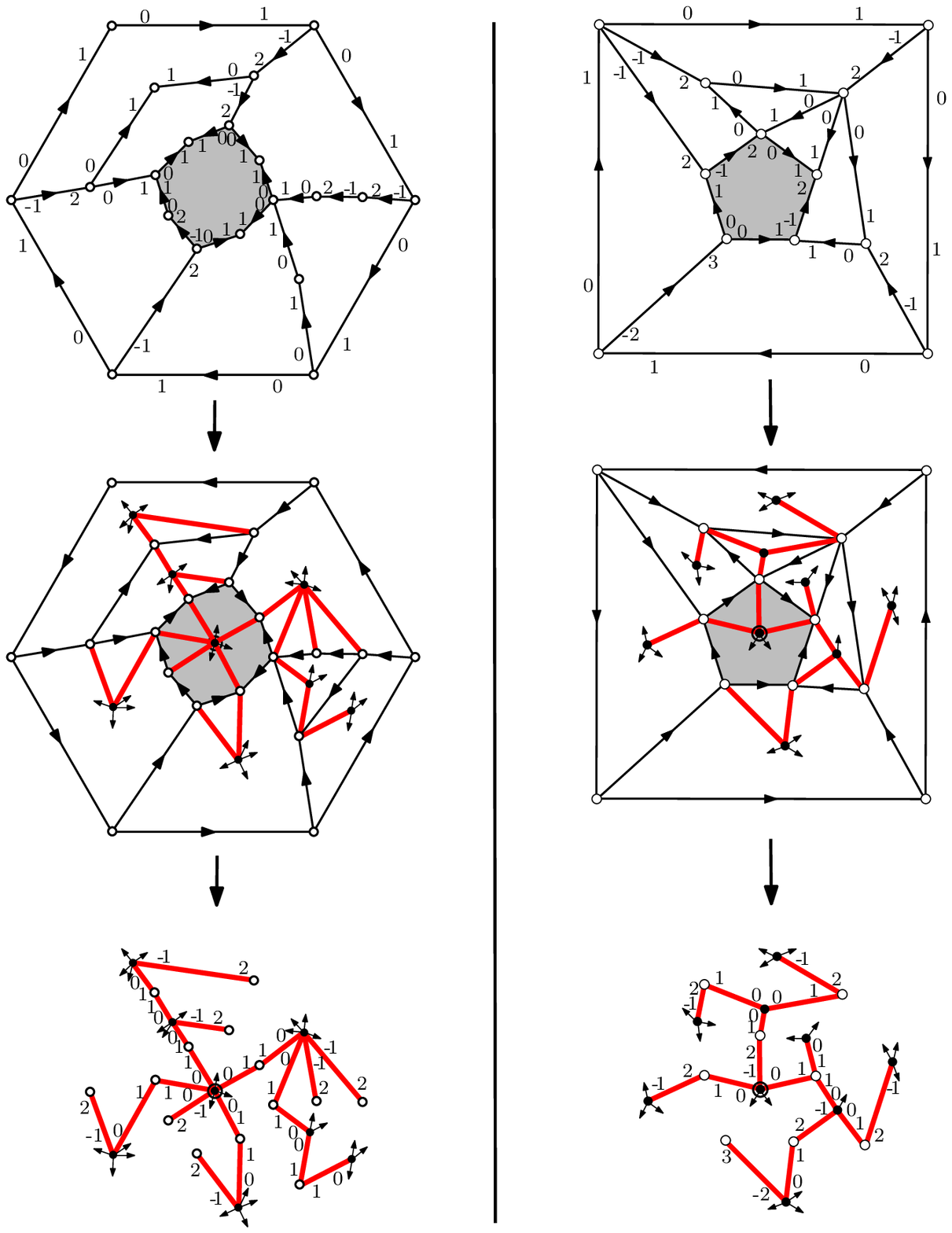}
\end{center}
\caption{Bijection for annular maps. Left: a bipartite annular map in $\cA\rsb$ with $\{b=2,r=3,s=4\}$ endowed with its suitable \bbm-orientation, and the associated $b$-dibranching mobile of type $(2r,2s)$;
to be compared with the left part of Figure~\ref{fig:bij_bip} ($b=2$, one root-face). Right: annular map in $\cA\pqd$ with $\{d=3,p=4,q=5\}$ endowed with its suitable \ddm-orientation, and the associated $d$-branching mobile of type $(p,q)$; to be compared with the left part of Figure~\ref{fig:bij_gen} ($d=3$, one root-face).}
\label{fig:bij-annular}
\end{figure}

\subsection{General case}\label{sec:gen_annular}
We now treat the case of general (not necessarily bipartite) maps.
In this subsection we fix positive integers $d,p,q$ with $p\leq q$.

\begin{Def} \label{def:ddmorient_annular}
Let $M$ be an annular map of type $(p,q)$ having no face of degree less than~$d$. A \emph{\ddm-orientation} of $M$ is an admissible \zb-biorientation such that every outgoing half-edge has weight 0, $-1$ or $-2$ and
\begin{enumerate}
\item[(i)] each inner edge has weight $d-2$,
\item[(ii)] each inner vertex has weight~$d$,
\item[(iii)] each non-root face $f$ has degree and weight satisfying $\deg(f)+\weight(f)=d$,
\item[(iv)] the inner root-face has degree $q$ and weight $p-q$.
\end{enumerate}
\end{Def}

Figure~\ref{fig:bij-annular} (top right) shows a \ddm-orientation for $d=3$.
Note that when $p=d$ every inner face satisfies $\deg(f)+\weight(f)=d$, in which
case we recover the definition of \ddm-orientations for plane maps of outer degree~$d$,
as given in Section~\ref{sec:bij_girth}.


\begin{thm}\label{thm:gen_annular}
Let $M$ be an annular map of type $(p,q)$ having no face of degree less than~$d$. Then, $M$ admits a \ddm-orientation if and only if $M$ is in $\cA\pqd$. In this case, there exists a unique suitable \ddm-orientation of $M$.
\end{thm}

\begin{Remark}\label{rk:bip-specialization-annular} If $d=2b$ and $M$ is a bipartite annular map in $\cA\pqd$, then the unique suitable \ddm-orientation of $M$ is obtained from its suitable \bbm-orientation by doubling the weight of every inner half-edge. 
\end{Remark}

The proof of Theorem~\ref{thm:gen_annular} (which extends
Theorem~\ref{thm:gen}) is delayed to Section~\ref{sec:proof}. 


\begin{Def}\label{def:dbranching_annular}
A \emph{$d$-branching mobile of type $(p,q)$} is a \zb-mobile with a marked black vertex called \emph{special vertex} such that half-edges incident to black vertices have weight 0, $-1$ or $-2$ and
\begin{enumerate}
\item[(i)] each edge has weight $d-2$,
\item[(ii)] each white vertex has weight~$d$,
\item[(iii)] each non-special black vertex $v$ has degree and weight satisfying \mbox{$\deg(v)+\weight(v)=d$,} 
\item[(iv)] the special vertex has degree $q$ and weight $p-q$.
\end{enumerate}
\end{Def}

The proof of the following claim is similar to the one used for Claim~\ref{claim:excess_bip}.
\begin{claim}\label{claim:excess_gen_annular}
Any $d$-branching mobile of type $(p,q)$ has excess $-p$.
\end{claim}

We now come to the main result of this subsection, which is the correspondence between the set  $\cA\pqd$ of annular maps and $d$-branching mobiles of type $(p,q)$. First of all, by Theorem~\ref{thm:gen_annular}, the set $\cA\pqd$ can be identified with the set of \ddm-oriented annular maps of type $(p,q)$ in $\wO$. Moreover, as in the bipartite case, it is easy to see from Theorem~\ref{thm:kmaster-bijections} that the master bijection $\Phi$ induces a bijection between this subset of $\wO$ and the set of $d$-branching mobiles of type $(p,q)$. We conclude:

\begin{theo}\label{thm:girthd_annular}
Annular maps in $\cA\pqd$
are in bijection with $d$-branching mobiles of type $(p,q)$. Moreover, each non-root face of degree $i$ in the map corresponds
to a non-special black vertex of degree $i$ in the mobile.
\end{theo}
Theorem~\ref{thm:girthd_annular} is illustrated in Figure~\ref{fig:bij-annular} (right).
The case $d=p$ in Theorem~\ref{thm:girthd_annular} corresponds to the bijection of Theorem~\ref{thm:girthd} 
where an inner face is marked. 

\begin{Remark} For $d=2b$ it is clear from Remark~\ref{rk:bip-specialization-annular} that the bijection of Theorem~\ref{thm:girth2b_annular} is equal to the specialization of the bijection of Theorem~\ref{thm:girthd_annular}, up to dividing the weights of the mobiles by two. 
\end{Remark}


\section{Counting results}\label{sec:count}
In this section we derive the enumerative consequences of the bijections described in the previous sections.

\subsection{Counting maps with one root-face}\label{sec:count_one_root_face}
In this subsection we give, for each positive integer~$d$, a system of equations specifying the generating function $F_{d}$ of
rooted maps of girth~$d$ and outer degree~$d$ counted according to the number of inner face of each degree.

We first set some notation. For any integers $p,q$ we denote by $[p\,..\,q]$ the set of integers $\{k\in\mathbb{Z},~p\leq k\leq q\}$. If $G(x)$ is a (Laurent) formal power series in $x$, we denote by $[x^k]G(x)$, the coefficient of $x^k$ in $G(x)$. For each non-negative integer $j$ we define the polynomial $h_j$ in the variables $w_1,w_2,\ldots$ by:
\begin{equation}\label{eq:def_hj}
h_j(w_1,w_2,\ldots):=[t^j]\frac{1}{1-\sum_{i>0}t^i w_i}=\sum_{r\geq 0}\sum_{\substack{i_1,\ldots,i_r>0\\ i_1+\cdots+i_r=j}}w_{i_1}\cdots w_{i_r}.
\end{equation}
 In other words, $h_j$ is the (polynomial) generating function of integer compositions of $j$ where the variable $w_i$ marks the number of parts of size~$i$. Note that $h_0=1$.

Let~$d$ be a positive integer. By Theorem~\ref{thm:girth2b} and Claim~\ref{claim:bij}, counting rooted plane maps of girth~$d$ and outer degree~$d$ reduces to counting $d$-branching mobiles rooted at an exposed bud. To carry out the latter task we simply write the generating function equation corresponding to the recursive decomposition of trees. We call \emph{planted $d$-branching mobile} a mobile with a dangling half-edge that can be obtained as one of the two connected components after cutting a $d$-branching mobile $M$ at the middle of an edge.
The weight of the dangling half-edge $h$ is called the \emph{root-weight}, and the vertex incident to $h$ is called the \emph{root-vertex}. Recall that the half-edges of a $d$-branching mobiles have weight in $[-2\,..\,d]$. For $j$ in $[-2\,..\,d]$, we denote by $\cW_j$ the family of planted $d$-branching mobiles of root-weight $d-2-j$. We denote by $W_j\equiv W_j(x_{d},x_{d+1},x_{d+2}\ldots)$ the generating function of $\cW_j$, where for $k\geq d$ the variable $x_k$ marks the black vertices of degree $k$. We now consider the recursive decomposition of planted mobiles and translate it into a system of equations characterizing the series $W_{-2},\ldots,W_d$.

Let $j$ be in $[-2\,..\,d-3]$, and let $T$ be a planted mobile in $\cW_{j}$. Since $d-2-j>0$, the root-vertex $v$ of $T$ is white, hence is incident to half-edges having positive weights. Let $e_1,\ldots,e_r$ be the edges incident to $v$. For all $i=1\ldots r$, let $T_i$ be the planted mobile obtained by cutting the edge $e_i$ in the middle ($T_i$ is the subtree not containing $v$), and let $\al(i)>0$ be the weight of the half-edge of $e_i$ incident to $v$ (so that $T_i$ is in $\cW_{\al(i)}$). Since the white vertex $v$ has weight~$d$, one gets the constraint $\sum_i \al(i)=j+2$. Conversely, any sequence of planted mobiles $T_1, \ldots,T_r$ in $\cW_{\al(1)},\ldots,\cW_{\al(r)}$ such that $\al(i)>0$ and $\sum_i \al(i)=j+2$ gives a planted mobile in $\cW_{j}$. Thus for all $j$ in $[-2\,..\,d-3]$,
$$
W_j=\sum_{r\geq 0}\sum_{\substack{i_1,\ldots,i_r>0\\ i_1+\cdots+i_r=j+2}}W_{i_1}\cdots W_{i_r}=h_{j+2}(W_1,\ldots,W_{d-1}).
$$
Note that the special case $W_{-2}=1$ is consistent with our convention $h_0=1$.
Observe also that $W_{-1}=h_1(W_1)=W_1$ whenever $d>1$.

Now let $j$ be in $[d-2\,..\,d]$, let $T$ be a planted mobile in $\cW_{j}$. Since $d-2-j \leq 0$,  the root-vertex $v$ of $T$ is black. If $v$ has degree $i$, then there is a sequence of $i-1$ buds and non-dangling half-edges incident to $v$. Each non-dangling half-edge $h$ has weight $\al\in\{0,-1,-2\}$, and cutting the edge containing $h$ gives a planted mobile in $\cW_\al$. Lastly, Condition (iii) of $d$-branching mobiles implies that the sum of weights of non-dangling half-edges is $d-\mathrm{deg}(v)-(d-2-j)=j+2-i$. Conversely, any sequence of buds and non-dangling half-edges satisfying these conditions gives a planted mobile in $\cW_{j}$.
Thus for all $j$ in $[d-2\,..\,d]$,
$$
W_j=[u^{j+2}]\sum_{i\geq d}x_iu^{i}(1+W_0+u^{-1}W_{-1}+u^{-2})^{i-1},
$$
where the summands $1$, $W_0$, $u^{-1}W_{-1}$ and $u^{-2}=u^{-2}W_{-2}$ in the parenthesis correspond respectively to the buds and non-dangling half-edges of weight 0, $-1$, $-2$ incident to $v$. We summarize:
\begin{theo}\label{thm:count_gen}
Let~$d$ be a positive integer, and let $F_{d}\equiv F_{d}(x_{d},x_{d+1},x_{d+2},\ldots)$ be the \gf of rooted
maps of girth~$d$ with outer degree~$d$, where each variable $x_{i}$ counts the inner faces of degree $i$. Then,
\begin{equation}\label{eq:systF}
F_{d}=W_{d-2}-\sum_{j=-2}^{d-3}W_j W_{d-2-j},
\end{equation}
where $W_{-2}=1,W_{-1},W_{0},\ldots,W_{d}$ are the unique formal power series satisfying:
\begin{equation}\label{eq:syst2}
\left\{
\begin{array}{ll}\ds
W_j=h_{j+2}(W_1,\ldots,W_{d-1}) & \textrm{for all } j \textrm{ in }[-2\,..\,d-3],\\[.1cm]
\ds W_j=[u^{j+2}]\sum_{i\geq d}x_iu^i(1+W_0+u^{-1}W_{-1}+u^{-2})^{i-1} & \textrm{for all } j \textrm{ in }[d-2\,..\,d],\\
\end{array}
\right.
\end{equation}
where the polynomials $h_j$ are defined by~\eqref{eq:def_hj}. In particular, for any finite set $\Delta\subset \{d,d+1,d+2,\ldots\}$, the specialization of $F_d$ obtained by setting $x_i=0$ for all $i$ not in $\Delta$ is algebraic (over the field of rational function in $x_i,i\in\Delta$).
\end{theo}

For $d=1$, Theorem~\ref{thm:count_gen} gives exactly the system of equations obtained by Bouttier, Di Francesco and Guitter in~\cite{Boutt}. Observe that for any integer $d\geq 2$ the series $W_{-1}$ and $W_1$ are equal, so the number of unknown series is $d+1$ in these cases. Moreover for $d\geq 1$ the series $W_d$ is not needed to define the other series $W_0,W_1,\ldots,W_{d-1}$. Lastly, under the specialization $\{x_d=x,\ x_i=0\ \forall i>d\}$ one gets $W_{d-1}=W_d=0$ and $W_{d-2}=x(1+W_0)^{d-1}$; in this case we recover the system of equations given in~\cite{BeFu10} for the generating function of rooted $d$-angulations of girth~$d$.

\begin{proof}
The fact that the solution of the system~\eqref{eq:syst2} is unique is clear. Indeed, it is easy to see that the series $W_{-1},W_{0},\ldots,W_{d}$ have no constant terms, and from this it follows that the coefficients of these series are uniquely determined by induction on the total degree.\\
We now prove~\eqref{eq:systF}. By Theorem~\ref{thm:girthd} and Claim~\ref{claim:bij} (first assertion) the series $F_{d}$ is equal to the generating function of $d$-branching mobiles with a marked exposed bud (where $x_k$ marks the black vertices of degree $k$). Moreover by the second assertion of Claim~\ref{claim:bij}, $F_d$ is equal to the difference between the generating function $B_d$ of $d$-branching mobiles with a marked bud, and the generating function $H_d$ of $d$-branching mobiles with a marked half-edge incident to a white vertex. Lastly, $B_d=W_{d-2}$ because $d$-branching mobiles with a marked bud identify with planted mobiles in $\cW_{d-2}$, and $H_d=\sum_{j=-2}^{d-3}W_j W_{d-2-j}$ because $d$-branching mobiles with a marked half-edge incident to a white vertex are in bijection (by cutting the edge) with ordered pairs $(T,T')$ of planted $d$-branching mobiles in $\cW_j\times \cW_{d-2-j}$ for some $j$ in $[-2\,..\,d-3]$.
\end{proof}

We now explore the simplifications occurring in the bipartite case.

\begin{theo}\label{thm:count_bip}
Let $b\geq 1$, and let $E_{b}\equiv F_{2b}(x_{2b},0,x_{2b+2},0,x_{2b+4}\ldots)$ be the \gf of rooted bipartite maps of girth $2b$ with outer degree $2b$, where each variable $x_{2i}$ marks the number of inner faces of degree $2i$. Then,
\begin{equation}\label{eq:systE}
E_{b}=V_{b-1}-\sum_{j=-1}^{b-2}V_j V_{b-j-1},
\end{equation}
where $V_{-1}=1,V_{0},\ldots,V_{b}$ are the unique formal power series satisfying:
\begin{equation}\label{eq:syst1_bis}
\left\{
\begin{array}{lll}\ds
V_j=h_{j+1}(V_1,\ldots,V_{b-1}) & \textrm{for all } j \textrm{ in }[-1\,..\,b-2],\\[.1cm]
\ds V_{j}=\sum_{i\geq b}x_{2i}\binom{2i-1}{i-j-1}(1+V_0)^{i+j} & \textrm{for all } j \textrm{ in } \{b-1,b\}.
\end{array}
\right.
\end{equation}
\end{theo}

Theorem~\ref{thm:count_bip} can be obtained by a direct counting of $b$-dibranching mobiles (which are simpler than $d$-branching mobiles). However in the proof below we derive Theorem~\ref{thm:count_bip} as a consequence of Theorem~\ref{thm:count_gen}. 

\begin{proof} Equations~\eqref{eq:systE} and~\eqref{eq:syst1_bis} are obtained respectively from~\eqref{eq:systF} and~\eqref{eq:syst2} simply by setting for all integer $i$, $x_{2i+1}=0$, $W_{2i}=V_i$, $W_{2i+1}=0$. Hence we only need to prove that the series $W_i$ defined by~\eqref{eq:systF} satisfy for all $i$, $W_{2i+1}(x_{2b},0,x_{2b+2},0,\ldots)=0$. This property holds because one can show that \emph{every monomial in the series $W_{2i+1}(x_{2b},x_{2b+1},x_{2b+2},\ldots),i\in\mathbb{Z}$ contains at least one variable $x_r$ with $r$ odd}, by a simple induction on the total degree of these monomials.
\end{proof}


\subsection{Counting maps with two root-faces}\label{sec:count_annular}
In this subsection we count rooted annular maps according to the face degrees and according to the two girth parameters. For positive integers $d,e,p,q$, we denote by $\cA\pqde$ the class of annular maps of type $(p,q)$ having non-separating girth at least~$d$ and separating girth at least~$e$. Recall that an annular map is \emph{rooted} if a corner is marked in each of the root-faces. We will now derive an expression for the generating functions $\G\pqde$ of maps obtained by rooting the annular maps in $\cA\pqde$.

\begin{theo}\label{thm:count_gen_annular}
For any positive integers $d,e,p,q$, the series $\G\pqde=\G\pqde(x_d,x_{d+1},\ldots)$ counting rooted annular maps of type $(p,q)$
with non-separating girth at least~$d$ and separating girth at least $e$ (where $x_k$ marks the number of non-root faces of degree~$k$) is 
\begin{equation}\label{eq:expC}
\G\pqde=\sum_{i=0}^{p-e}\sum_{j=0}^{q-e}\mathbf{1}_{i+j\equiv p+q \ (\mathrm{mod}\ 2)}\,\frac{2\beta(p,i,e)\beta(q,j,e)}{p+q-i-j}(1+W_0)^{(p+q-i-j)/2}W_{-1}\ \!\!^{i+j},
\end{equation}
where the formal power series $W_{-1},W_{0},\ldots,W_{d}$ are specified by~\eqref{eq:syst2}, and where
$$\displaystyle \beta(p,i,e):=\frac{p!}{i!\lfloor \tfrac{p-i-e}{2} \rfloor!\lfloor\tfrac{p-i+e-1}{2} \rfloor!}.$$
\end{theo}

\begin{proof} The proof has two parts. First we will use the bijection obtained in Section~\ref{sec:bij_girth_annular} in order to characterize the series $\G\pqde$ in the case $e=p$. Then we will treat the case of an arbitrary separating girth $e\geq p$ (in the case $e< p$, $\G\pqde=0$).

By definition the series $\G\pqdp$ counts maps obtained by rooting the annular maps in $\cA\pqd\equiv\cA\pqdp$. Let $\mX$ be the set of  $d$-branching mobiles of type $(p,q)$ with a marked corner at the special vertex.   
By Theorem~\ref{thm:girthd_annular} maps in $\cA\pqd$  are in bijection with the $d$-branching mobiles of type $(p,q)$. 
Hence it is easy to see from the definition of the master bijection $\Phi$ that maps obtained by marking a corner in the inner root-face of a map in  $\cA\pqd$ are in bijection with the mobiles in $\mX$.  Hence, maps obtained by rooting the annular maps in $\cA\pqd$ are in $p$-to-1 correspondence with the mobiles in $\mX$. It remains to count the mobiles in $\mX$.
Let $M$ be a mobile in $\mX$ and let $v_0$ be the special vertex. The vertex $v_0$ is black and is incident to a sequence of $q$  buds and non-dangling half-edges.  Each non-dangling half-edge $h$ incident to $v_0$ has a weight $\al$ in $\{0,-1,-2\}$, and cutting the edge containing $h$ gives a planted mobile in $\cW_\al$. Moreover the total weight of the non-dangling half-edges incident to $v_0$ is $p-q$. Conversely any sequence of $q$  buds and non-dangling half-edges satisfying these conditions gives a mobile in $\mX$. This bijective decomposition of the mobiles in $\mX$  gives an expression for the generating function of $\mX$, or equivalently for the series $\G\pqdp$:
\begin{equation}\label{eq:Npqd}
\G\pqdp=p\cdot[u^{p-q}](1+W_0+u^{-1}W_{-1}+u^{-2})^{q}.
\end{equation}
Here the summands $1$, $W_0$, $u^{-1}W_{-1}$ and $u^{-2}=u^{-2}W_{-2}$ correspond respectively to the buds and to the half-edges of weight 0, $-1$, $-2$ incident to the special vertex.

We will now derive an expression for the series $\G\pqde$ when $e\geq p$. We first partition the set $\cA\pqde$. 
For $a\leq \mathrm{min}(p,q)$, we denote by $\wt{\cA}\pqda$ the class of annular maps of type $(p,q)$ having non-separating girth at least~$d$, and separating girth \emph{exactly} $a$. Let $\wt{\G}\pqda\equiv \wt{\G}\pqda(x_{d},x_{d+1},\ldots)$ be the \gf counting rooted annular maps from $\wt{\cA}\pqda$. 
Clearly, $\cA\pqde=\uplus_{a=e}^{\min(p,q)}\wt{\cA}\pqda$, so that $\G\pqde=\sum_{a=e}^{\mathrm{min}(p,q)}\wt{\G}\pqda$.

For $p\leq q$ we denote by $\wh{\cA}\pqd$ the subfamily of $\cA\pqd$ where the unique separating cycle
of length $p$ is the contour of the outer face, and we denote by $\wh{\G}\pqd\equiv\wh{\G}\pqd(x_d,x_{d+1},\ldots)$ the \gf counting rooted annular maps from $\wh{\cA}\pqd$. Let $M$ be a rooted annular map from $\cA\pqd$. It is easy to see that there exists a unique \emph{innermost} cycle $C$ of length $p$ enclosing the inner root-face (i.e., any other cycle of length $p$ enclosing the inner root-face also encloses $C$). The cycle $C$ is simple, and after distinguishing one of the $p$ vertices on $C$, one can identify the part of $M$ outside and inside $C$ as rooted annular maps $M_1$ and $M_2$ from $\cAppd$ and $\wh{\cA}\pqd$ respectively (the marked vertex of $C$ gives the marked corners of the inner root-face of $M_1$ and the outer root-face of $M_2$). This bijective decomposition of $M$ gives $ p\,\G\pqdp=\G\ppdp \cdot \wh{\G}\pqd$, or equivalently,
$$
\wh{\G}\pqd=p\, \frac{\G\pqdp}{\G\ppdp}.
$$
Now, given a rooted annular map $M$ from $\wt{\cA}\pqda$ with root-faces $f_1$ and $f_2$, we consider the outermost and innermost cycles $C_1$ and $C_2$ of length $a$ separating $f_1$ from $f_2$. By distinguishing some vertices $v_1$ and $v_2$ from $C_1$ and $C_2$ and cutting $M$ along $C_1$ and $C_2$ one obtains three rooted annular maps respectively from $\wh{\cA}\apd$, $\cAaad$, and $\wh{\cA}\aqd$ 
(the root-corner in the root-face enclosed by $C_1$ is the one incident to $v_1$, 
and the root-corner in the root-face enclosed by $C_2$ is the one incident to $v_2$). This decomposition yields
$$
a^2\wt{\G}\pqda=\wh{\G}\apd\G\aada\wh{\G}\aqd=a^2\frac{\G\apda\G\aqda}{\G\aada}.
$$
By \eqref{eq:Npqd}, 
$$\G\apda=a \sum_{i=0}^{p-a}\gamma(p,i,a)\, (1+W_0)^{(p+a-i)/2}{W_{-1}}^i$$
where $\displaystyle \gamma(p,i,a)=\mathbf{1}_{p-i\equiv a \ (\mathrm{mod}\ 2)}\frac{p!}{i!(\tfrac{p-i-a}{2})!(\tfrac{p-i+a}{2})!}$. Hence, for $a\leq \min(p,q)$,
$$
\wt{\G}\pqda=\sum_{i=0}^{p-a}\sum_{j=0}^{q-a}a\,\gamma(p,i,a)\gamma(q,j,a)\,(1+W_0)^{(p+q-i-j)/2}{W_{-1}}^{i+j}.
$$
To conclude we have $$\G\pqde=\sum_{i=0}^{p-e}\sum_{j=0}^{q-e}\ \sum_{a=e}^{\min(p-i,q-j)}a\,\gamma(p,i,a)\gamma(q,j,a)\,(1+W_0)^{(p+q-i-j)/2}{W_{-1}}^{i+j}.$$
Moreover, the identity $$\sum_{a=e}^{\min(p-i,q-j)}a\,\gamma(p,i,a)\gamma(q,j,a)=\mathbf{1}_{i+j\equiv p+q \ (\mathrm{mod}\ 2)}\cdot\frac{2\beta(p,i,e)\beta(q,j,e)}{p+q-i-j}$$ can be obtained by a simple induction on $e$, decreasing from the base case $e=\min(p-i,q-j)$. 
Indeed, if $e\equiv p-i \equiv q-j\, (\mathrm{mod}\ 2)$, one has 
$$\frac{2}{p+q-i-j}\left(\beta(p,i,e)\beta(q,j,e)-\beta(p,i,e+1)\beta(q,j,e+1)\right)=e\,\gamma(p,i,e)\gamma(q,j,e),$$
and $\beta(p,i,e-1)=\beta(p,i,e)$, $\beta(q,j,e-1)=\beta(q,j,e)$.\\
This completes the proof of Theorem~\ref{thm:count_gen_annular}.
\end{proof}

As a corollary we obtain the following universal asymptotic behavior for the number of $n$-faces
rooted maps with restrictions on the girth and face-degrees (as mentioned in the introduction, a result
of a similar flavor was established by Bender and Canfield for bipartite maps~\cite{BeCa94}, with
no control on the girth): 

\begin{cor}
For any non-empty finite set $\Delta\subset \{d,d+1,d+2,\ldots\}$, the specialization of $\G\pqde$ obtained by setting $x_i=0$ for all $i$ not in $\Delta$ is algebraic (over the field of rational function in $x_i,i\in\Delta$). Moreover, there exist computable constants $\kappa,\gamma$ depending on~$d$ and $\Delta$ such that if $\Delta$ contains at least one even integer (resp. $\Delta$ contains only odd integers) the number $c_{d,\Delta}(n)$ of rooted plane maps of girth at least~$d$ with $n$ faces (resp. $2n$ faces), all of them having degrees in $\Delta$, is asymptotically equivalent to $\kappa\, n^{-5/2}\,\gamma^n$.
\end{cor}

\begin{proof}
The algebraicity of the series $W_{-2}\ldots W_d$ is obvious from the system \eqref{eq:syst2}, hence $\G\pqde$ is algebraic (as soon as $\Delta$ is finite). 

We now consider the specialization of the series $W_{-2}\ldots W_d$ and $\G\pqde$ obtained by replacing all the variable $x_i,i\in\Delta$ by $t$ (and setting the other variables $x_i$ to 0). We suppose first that $\Delta$ contains at least one even integer. 
Given the form of the system \eqref{eq:syst2}, the Drmota-Lalley-Wood theorem~(see~\cite[VII.6]{fla}) implies that the series $W_i,i\in[-1..d-1]$ all have the same ``square-root type'' singularity at their unique dominant singularity $\gamma$. Therefore, the same applies to $\G\pqde$, implying $[t^n]\G\pqde \sim \alpha \, n^{-3/2}\gamma^n$ for some computable constants $\alpha,\gamma$ (depending on $p,q,d,e,\Delta$). Observe that $\frac{1}{q}[t^n]\G\pqdd$ counts rooted plane maps of girth at least~$d$, with a root-face of degree $p$, a marked inner face of degree $q$, and $n$ additional inner faces having degrees in $\Delta$. Therefore, 
$$(n+1)\, c_{d,\Delta}(n+2)=\sum_{p,q\in\Delta} \frac{1}{q}[t^{n}]\G\pqdd.$$ 
This gives the claimed asymptotic form of $c_{d,\Delta}(n)$ when $\Delta$ contains an even integer. 

If $\Delta$ contains only odd integers, one has to deal with the \emph{periodicity} of the series $W_i,i\in[-1..d-1]$ (one can easily check that $[t^i]W_j=0$ unless $i\equiv j \ (\mathrm{mod}\ 2)$, and $[t^n]\G\pqde=0$ unless $n\equiv p+q \ (\mathrm{mod}\ 2)$). However, up to using a variable $z=t^2$, one can still use the Drmota-Lalley-Wood theorem to prove the asymptotic form $[t^{2n}]\G\pqde \sim \alpha\, n^{-3/2}\,\gamma^n$ for $p,q\in\Delta$, from which the stated result follows.
\end{proof}

As in Section~\ref{sec:count_one_root_face}, the generating functions have a simpler expression in the bipartite case:

\begin{theo}\label{thm:count_bip_annular}
For $b,c,r,s$ positive integers, let $B\rsbc\equiv \G\Brsbc(x_{2b},0,x_{2b+2},0,\ldots)$ be the \gf of rooted annular bipartite maps from $\cA\Brsbc$,
where each variable $x_{2i}$ marks the number of inner faces of degree $2i$. Then,
\begin{equation}\label{eq:expB}
B\rsbc=\frac{4rs}{r+s}\binom{2r-1}{r-c}\binom{2s-1}{s-c}(1+V_0)^{r+s},
\end{equation}
where $V_{0},\ldots,V_{b}$ are given by~\eqref{eq:syst1_bis}. 
\end{theo}

\begin{proof}
Again the expression can either be obtained by a direct counting of $b$-dibranching mobiles (which are simpler than $d$-branching mobiles), or just by specializing the expression in Theorem~\ref{thm:count_gen_annular}.
As we have seen in the proof of Theorem~\ref{thm:count_bip},
when $x_{2i+1}=0$ for all integer $i$, then $W_{r}=0$ for all odd $r\in[-1..d]$
and the series $V_i:=W_{2i}$ satisfy~\eqref{eq:syst1_bis}. Since $W_{-1}=0$, there remains
only the initial term ($i=0$ and $j=0$) in the expression~\eqref{eq:expC} of $\G\Brsbc(x_{2b},0,x_{2b+2},0,x_{2b+4}\ldots)$,
which gives~\eqref{eq:expB}.
\end{proof}

\subsection{Exact formula for simple bipartite maps}
In this subsection, we obtain a closed formula for the number of rooted simple bipartite maps 
from the case $b=2$ of Theorem~\ref{thm:count_bip_annular}.
\begin{prop}[simple bipartite maps]\label{prop:rootedsimple}
Let $k\geq 2$, and let $n_2,n_3,\ldots,n_k$ be non-negative integers not all equal to zero.
The number of rooted simple bipartite maps
with $n_i$ faces of degree $2i$ for all $i\in\{2,3,\ldots, k\}$ is
\begin{equation}\label{eq:rootedsimple}
2\frac{(e+n-3)!}{(e-1)!}\prod_{i=2}^k\frac1{n_i!}\binom{2i-1}{i-2}^{n_i},
\end{equation}
where $n=\sum_in_i$ is the number of faces, and $e=\sum_iin_i$ is the number of edges.
\end{prop}
\begin{proof}
Let $a(n_2,\ldots,n_k)$ be the number of rooted simple bipartite maps with $n_i$ faces of degree $2i$ for $i\in[2..k]$. If $\sum_in_i=1$ (i.e., the map has a single face), Formula~\eqref{eq:rootedsimple} gives the $e$th Catalan number, which indeed counts rooted plane trees with $e$ edges. We now suppose $\sum_in_i\geq 2$ and consider integers $r,s$ such that $\bn_i:=n_i-\mathbf{1}_{i=r}-\mathbf{1}_{i=s}$ is non-negative for all $i\in[2..k]$. Let $N$ be the number of rooted annular maps of type $(2r,2s)$ with $\bn_i$ non-root faces of degree $2i$ for $i\in[2..k]$. 
Counting in two different ways rooted annular maps with a third root (a marked corner) placed anywhere, we obtain
$2eN= 4r s n_{r}n_{s}a(n_2,\ldots,n_k)$ if $r\neq s$ and $2eN= 4rsn_{r}(n_{r}-1)a(n_2,\ldots,n_k)$ if $r=s$.
Thus it remains to prove 
\begin{equation}\label{eq:N}
N=4rs\frac{(e+n-3)!}{e!}\prod_{i=2}^k\frac1{\bn_i!}\binom{2i-1}{i-2}^{n_i}.
\end{equation}
By Theorem~\ref{thm:count_bip_annular}, $N$ is the coefficient $[x_2^{\bn_2}\ldots x_k^{\bn_k}]$ of the series
$$B_{r,s}^{(2,2)}=\frac{4rs}{r+s}\binom{2r-1}{r-2}\binom{2s-1}{s-2}R^{r+s},$$
where the series $R\equiv 1+V_0$ is specified by $R=1+\sum_{i\geq 2}x_{2i}\binom{2i-1}{i-2}R^{i+1}$.
The Lagrange inversion formula yields
\begin{equation}\label{eq:Ra}
[x_2^{\bn_2}\ldots x_k^{\bn_k}]R^a=a\frac{(\sum_i(i+1)\bn_i+a-1)!}{(\sum_ii\bn_i+a)!\bn_2!\ldots\bn_k!}\prod_{i=2}^k\binom{2i-1}{i-2}^{\bn_i},
\end{equation}
which gives \eqref{eq:N}.
\end{proof}

\begin{Remark}
With some little efforts, the proof above can be made bijective. Indeed it is not very hard to obtain the expression \eqref{eq:Ra} of the coefficients of $R^a$ bijectively starting from the combinatorial description of the 2-dibranching mobiles.
\end{Remark}


\subsection{Counting loopless maps}
In this subsection, we focus on the case $d=2$ of Theorem~\ref{thm:count_gen} and show how to deduce from it the formula given in \cite{WaLe75} (where it is obtained by a substitution approach) for the number of rooted loopless maps with $n$ edges.
First observe that, up to collapsing the root-face of degree 2 into an edge, rooted maps in $\mC_2$ identify with rooted loopless maps with at least one edge (without constraint on the degree of the root-face). Hence, the multivariate series $F_2$ counts rooted loopless maps with at least one edge, where $x_i$ marks the number of faces of degree $i$. We now consider the specialization $x_i=t^i$ in $F_2$, which gives the \gf of rooted loopless maps with at least one edge counted according to the number of half-edges. 
This series is defined by the system of equations~\eqref{eq:syst2} in the case $d=2$, under the specialization $x_i=t^i$.  Using the notation $R:=1+W_0$ and $S:=W_{-1}=W_1$, this system becomes
\begin{equation}\nonumber 
F_2(t^2,t^3,t^4,\ldots)=R-1-S^2-t\,B_3, \qquad  R= 1+t\,B_1,\qquad S= t\,B_2, 
\end{equation}
where $B_k=[u^k]B$ and 
$$B=\sum_{i\geq 0}t^i(uR+S+u^{-1})^{i}.$$

Now we observe that $B_k$ is the series of Motzkin paths ending at height $k$ where up steps, horizontal steps, and down steps have respective weights $tR$, $tS$, and~$t$. The paths ending at height $0$ are called \emph{Motzkin bridges}, and have generating function $B_0$. The paths ending at height $0$ and having non-negative height all the way are called \emph{Motzkin excursions}, and we denote by $M$ their generating function. It is a classical
exercise to show the following identities:
$$
(i)~ B_k=B_0\,(tRM)^k,~~~~    (ii)~ M=1+tSM+t^2RM^2,~~~~
(iii)~ B_0=1+tSB_0+2t^2RMB_0.
$$
In particular $(i)$ gives
$$
(iv)~ R=1+t^2B_0MR,~~~~\textrm{ and }~~~(v)~ S=t^3B_0M^2R^2.
$$
So we have a system of four equations $\{(ii),(iii),(iv),(v)\}$ for the unknown series $\{M,B_0,R,S\}$, and this system has clearly a unique power series solution.
With the help of a computer algebra system, one can extract the first coefficients and then guess and check that the solution is $\{M=\alpha,\ B_0=\alpha^2,\ R=\alpha,\ S=t^3\alpha^6\}$, where the series $\alpha\equiv \alpha(t)$ is specified by $\alpha=1+t^2\alpha^4$. Hence,  $F_2(t^2,t^3,\ldots)=\alpha^2(2-\alpha)-1$. We summarize:

\begin{prop}
Let $c_n$ be the number of rooted loopless maps with $n$ edges and let $C(t)=\sum_{n\geq 0}c_nt^n$.
Then, $C(t)=\alpha^2(2-\alpha)$, where $\alpha\equiv\alpha(t)$ is the unique formal power series satisfying $\alpha=1+t\alpha^4$.
Hence, by the Lagrange inversion formula,
\begin{equation}\label{eq:rooted_loopless_edges}
c_n=\frac{2(4n+1)!}{(n+1)!(3n+2)!}.
\end{equation}
\end{prop}
Formula~\eqref{eq:rooted_loopless_edges} was already obtained in~\cite{WaLe75} using a substitution approach.
The sequence $\tfrac{2(4n+1)!}{(n+1)!(3n+2)!}$ appears recurrently in combinatorics, for instance it also counts
rooted simple triangulations with $n+3$ vertices~\cite{T62a,Poulalhon:triang-3connexe+boundary}, and intervals in the $n$th Tamari lattice~\cite{ch06,BeBo07}.\\

\section{Special cases $b=1$ and $d=0,1,2$.}\label{sec:special}
In this section, we take a closer look at the bijections given in Section~\ref{sec:bij_girth} in the particular cases $b=1$ and $d=1,2$. We also explain how to include the case $d=0$.

\subsection{Case $b=1$ (general bipartite maps) and relation with~\cite{Sc97}}
Let  $\mB$ be the class of bipartite plane maps of outer degree~2. Note that maps in $\mB$ have girth 2 (since bipartite maps cannot have cycles of length 1).  Moreover, $\mB$ can be identified with  the class of  bipartite maps with a marked edge (since the root-face of degree 2 can be collapsed into a marked edge). The case $b=1$ of Theorem~\ref{thm:girth2b} (illustrated in Figure~\ref{fig:loopless2}) gives a bijection between the class $\mB$, and the class of 1-dibranching mobiles.  We now take a closer look at this bijection and explain its relation with~\cite{Sc97}. In~\cite{Sc97}  Schaeffer obtained a bijection for \emph{Eulerian maps} (maps with vertices of even degree) with a marked edge. From the above remarks it follows that the class of Eulerian maps with a marked edges can be identified with the class $\mB$ via duality.

Observe from Figure~\ref{fig:b-dibranching-edges} that  1-dibranching mobiles have only two types of edges, and that their weights are redundant. Moreover all the white vertices are leaves. Hence, it is easy to see that the class of 1-dibranching mobiles identifies with the class of (unweighted) bicolored plane trees such that white vertices are leaves, and any black vertex adjacent to $\ell$ white leaves has degree $2+2\ell$. These are exactly the \emph{blossoming trees} defined by Schaeffer in~\cite{Sc97} (the white leaves are called ``stems'' there). Moreover the bijection of Schaeffer coincides with ours via duality: to obtain the map from the tree, the closure operations described respectively in Proposition~\ref{prop:closure} and in~\cite{Sc97} are the same.




The following formula (originally due to Tutte~\cite{T62b}) for the number $b[n_1,\ldots,n_k]$ of rooted bipartite maps with $n_i$ faces of degree $2i$ for $1\leq i\leq k$ can be obtained by counting blossoming trees (i.e., 1-dibranching mobiles) as done by Schaeffer in~\cite{Sc97}:
\begin{equation}\nonumber
b[n_1,\ldots,n_k]=2\frac{e!}{v!}\prod_{i=1}^k\frac1{n_i!}\binom{2i-1}{i-1}^{n_i},
\end{equation}
where $e=\sum_iin_i$ and $v=2+e-\sum_in_i$.

\subsection{Case $d=2$ (loopless maps) and relation with \cite[Thm.~2.3.4]{Schaeffer:these}}\label{sec:loopless}
We call \emph{edge-marked loopless map} a loopless planar maps with a marked edge. It is clear that the class $\mC_2$ (plane maps of girth $2$ and outer degree $2$) can be identified with the  class of edge-marked loopless maps  (since the root-face of degree 2 can be collapsed into a marked edge). Hence, for $d=2$, Theorem~\ref{thm:girthd} gives a bijection between edge-marked loopless maps and 2-branching mobiles. Some cases of this bijection are represented in Figure~\ref{fig:loopless2}.

\fig{width=\linewidth}{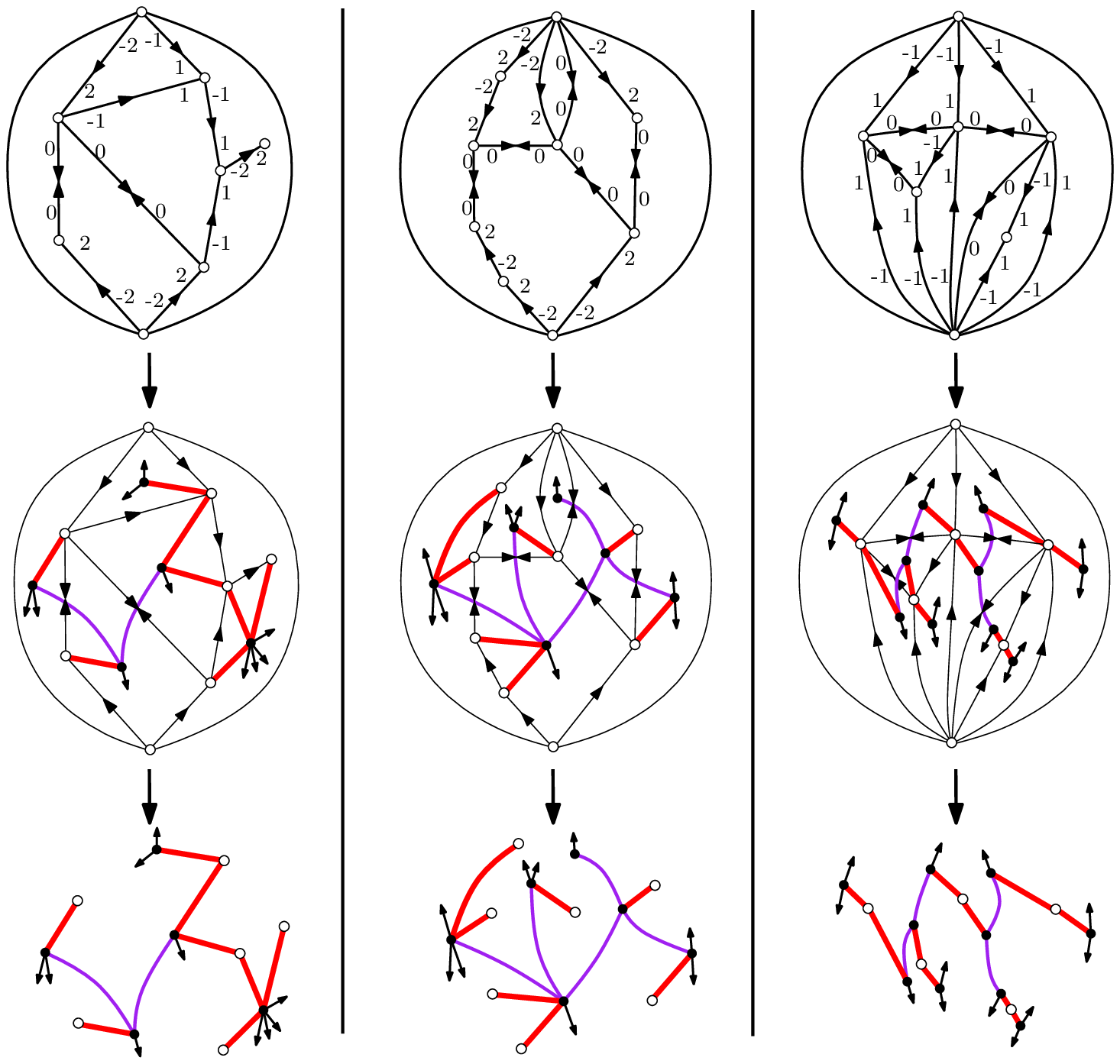}{Bijection for $d=2$ on 3 examples. 
The example in the middle column is bipartite hence gives a 1-dibranching mobile (in this case all the white vertices are leaves). The example in the right column has all its inner faces of degree~3 (in this case all the white vertices have degree $2$).}

Observe from Figure~\ref{fig:d-branching-edges} that  2-branching mobiles have only three types of edges, and that their weights are redundant. Up to forgetting these weights, the 2-branching mobiles are the bicolored plane trees such that there is no white-white edges, white vertices have degree 1 or 2, and black vertices adjacent to $\ell$ white leaves are incident to a total of $\ell+2$ buds or black-black edges. 
We now consider the specialization of our bijection for $d=2$ to triangulations (right of Figure~\ref{fig:loopless2}), and its relation with the bijection described by Schaeffer in~\cite[Thm.~2.3.4]{Schaeffer:these} for \emph{bridgeless cubic maps} (these are the dual of loopless triangulations). By the preceding remarks, Theorem~\ref{thm:girthd} yields a bijection between edge-marked loopless triangulations and the mobiles with the following properties: there are no white-white edges, every white vertex has degree 2, and  every black vertex has degree 3 and is adjacent to a unique white vertex. Clearly, these mobiles identify with the (unicolored) binary trees endowed with a perfect matching of the inner nodes. These are exactly the \emph{blossoming trees} shown to be in bijection with bridgeless cubic maps in~\cite[Thm.~2.3.4]{Schaeffer:these} (see also~\cite{PS03a}). Moreover, the bijection in~\cite[Thm.~2.3.4]{Schaeffer:these} coincides with ours via duality: to obtain the map from the tree, the closure operations described respectively in Proposition~\ref{prop:closure} and in~\cite[Thm.~2.3.4]{Schaeffer:these} are the same.

\subsection{Case $d=1$ (general maps) and relation with~\cite{Boutt}.}
The case $d=1$ of Theorem~\ref{thm:girthd} gives a bijection between the class $\mC_1$  and 1-branching mobiles. By definition, $\mC_1$ is the class of plane maps of girth $1$ and outer degree $1$. Hence, $\mC_1$ is the class of plane maps without girth constraint 
 such that the root-face is a loop. Note that this class can be identified with the class of rooted planar maps (indeed the root-face can be collapsed and thought as simply marking a corner). We now take a closer look at the bijection between  the class $\mC_1$  and 1-branching mobiles, and its relation with~\cite{Boutt}.
   In \cite{Boutt} Bouttier, Di Francesco and Guitter obtained a bijection for \emph{$1$-legged maps}, that is, planar maps with a marked vertex of degree $1$. Observe that the class of $1$-legged maps identifies with the class $\mC_1$ by duality (the marked vertex of degree 1 becomes a marked face of degree 1 via duality).

\begin{figure}
\begin{center}
\includegraphics[width=\linewidth]{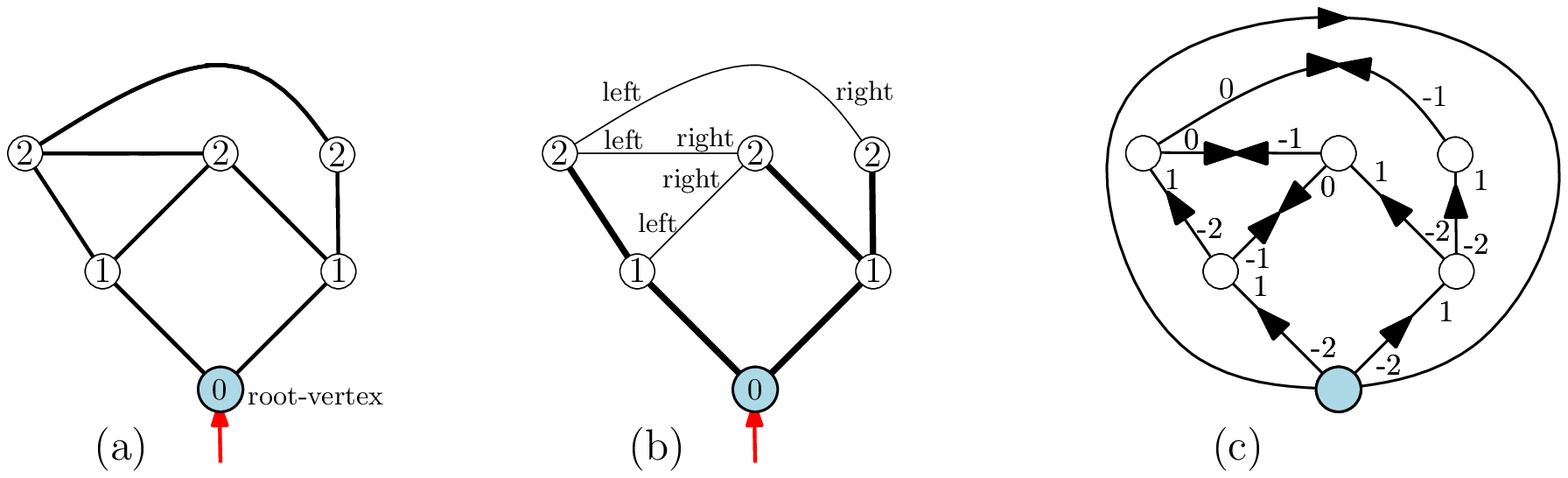}
\end{center}
\caption{(a) A rooted map $M$ and the root-distances. (b) The rightmost BFS-tree (thick lines) and the \emph{left} and \emph{right} half-edges. (c) The map in $\mC_1$ corresponding to $M$ endowed with its suitable $1/(-1)$-orientation.}
\label{fig:BFS-tree}
\end{figure}

\begin{figure}
\begin{center}
\includegraphics[width=\linewidth]{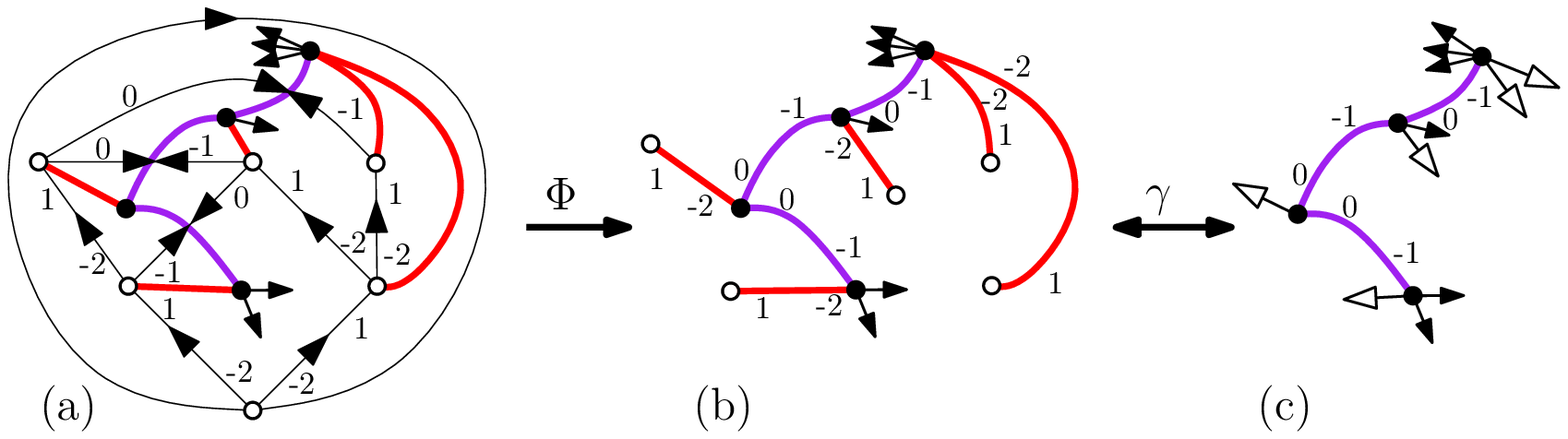}
\end{center}
\caption{
Bijection in the case $d=1$ and its relation with~\cite{Boutt}. (a) The bijection $\Phi$. (b) The resulting 1-branching mobile $B$. (c) The well-charged tree $\gamma(B)$.}
\label{fig:bij_d=1}
\end{figure}

We will first characterize the suitable $1/(-1)$-orientations. Let $M$ be a rooted map. We call \emph{root-distance} of a vertex $v$ the minimal length of the paths joining the root-vertex and $v$. A spanning tree of $M$ is a \emph{BFS-tree} (or \emph{bread-first-search tree}) if the root-distance of any vertex is the same in the map and in the tree. The root-distances and a BFS-tree are shown in Figure~\ref{fig:BFS-tree}. Let $T$ be a BFS-tree and let $e$ be an edge not in $T$. The edge $e$ creates a cycle with $T$ which separates two regions of the plane. We call \emph{left-to-right orientation} of $e$ the orientation such that the region on the left of $e$ contains the root-face. The outgoing and ingoing half-edges of $e$ are then called \emph{left} and \emph{right} half-edges. It is easy to see (see e.g.~\cite{OB:covered-maps}) that there exists a unique BFS-tree, called \emph{rightmost BFS-tree}, such that the root-distance does not decrease along edges not in $T$ traversed left-to-right. The following characterization of suitable $1/(-1)$-orientations is illustrated in Figure~\ref{fig:BFS-tree}.
\begin{prop}\label{prop:ruleBFS}
Let $M$ be a map in $\mC_1$ and let $T$ be its rightmost BFS-tree. Then, the suitable $1/(-1)$-orientation of $M$ is obtained as follows:
\begin{itemize}
\item Every edge in $T$ is 1-way, oriented from parent to child with weight $-2$ on the outgoing half-edge and weight 1 on the ingoing half-edge.
\item Every inner edge $e$ not in $T$ is 0-way with weight 0 and $-1$ on the half-edges. The weight $0$ is given to the left half-edge if the root-distance of the two endpoints of $e$ is the same and to the right half-edge otherwise. 
\end{itemize}
\end{prop}


We omit the (easy but tedious) proof of Proposition~\ref{prop:ruleBFS}. We now examine 1-branching mobiles and their relation with the \emph{well-charged trees} considered in~\cite{Boutt}. A \emph{charged tree} is a plane tree with two types of dangling half-edges called \emph{white arrows} and \emph{black arrows}. The \emph{charge} of a subtree $T'$ is the number of white arrows minus the number of black arrows in $T'$. A \emph{well-charged tree} is a charged tree such that cutting any edge gives two subtrees of charge 0 and -1 respectively. Now, observe from Figure~\ref{fig:d-branching-edges} that there are only two types of edges in 1-branching mobiles: black-white with weights
$(-2,1)$ or black-black with weights $(0,-1)$. It is easily seen that 1-branching mobiles are the mobile with these two type of edges such that white vertices are leaves and each black vertex $v$ has degree and weight satisfying $\deg(v)+\weight(v)=1$. 
For a 1-branching mobile $B$, we denote by $\gamma(B)$ the charged tree obtained by replacing white leaves and buds respectively by white arrows and black arrows. The mapping $\gamma$ is represented in Figure~\ref{fig:bij_d=1}. It is easy to check that for any black-black edge $e$ of $B$, the charges of the subtrees obtained by deleting the edge $e$ from $\gamma(B)$ are equal to the weights of the half-edges of $e$ incident to these subtrees. From this observation it easily follows that $\gamma$ is a bijection between 1-branching mobiles and well-charged trees. 

In~\cite{Boutt} a bijection was described between $1$-legged maps and well-charged trees. This bijection actually coincide with ours via duality (and the identification $\gamma$ between 1-branching mobiles and well-charged trees). Indeed 
 to obtain the map from the tree, the closure operations, described respectively in Proposition~\ref{prop:closure} and in~\cite{Boutt}, are the same
 (and the opening operations, to get the tree from the map, rely in the same way on the 
 rightmost BFS tree). 

\subsection{The case $d=0$ and relation with~\cite{BDFG:mobiles}.}\label{sec:d=0}
We show here that a slight reformulation of our bijections allows us to include the case $d=0$, thereby recovering a bijection obtained by Bouttier, Di Francesco and Guitter in~\cite{BDFG:mobiles}.

We call \emph{plane maps of outer degree 0} a planar map with a marked vertex called \emph{outer vertex}. Any face, any edge and any non-marked vertex of such a map is called \emph{inner}. A biorientation of a plane map of outer degree 0 is called \emph{accessible} if every inner vertex can be reached from the outer vertex, it is called \emph{minimal} if every directed simple cycle has the outer vertex strictly on its left, and it is called \emph{admissible} if every half-edge incident to the outer vertex is outgoing. We then say that a biorientation of a plane map of outer degree 0 is \emph{suitable} if it is minimal, admissible and accessible. 
We now reformulate the definition of \ddm-orientations so as to include the case $d=0$: these are the admissible biorientations of plane maps of outer degree~$d$ with inner and outer half-edges having weights in $d\cup \{1,2,3\ldots\}$ and $\{-2,-1,0\}\backslash \{d\}$ respectively, and satisfying the conditions (i),(ii),(iii) of Definition~\ref{def:ddmorient} (hence the definition is unchanged for $d>0$). 

For a plane map of outer degree 0, we consider the \emph{root-distance} $D(v)$ of each vertex $v$ (its graph distance to the outer vertex) and the \emph{geodesic biorientation}, that is, the biorientation where an edge $\{v,v'\}$ with $D(v')=D(v)$ is oriented $0$-way with weight $-1$ on each half-edge, while an edge $\{v,v'\}$ with $D(v')=D(v)+1$ is oriented $1$-way toward $v'$ with weight $-2$ on the outgoing half-edge and 0 on the ingoing half-edge. 
\begin{prop}
Theorem~\ref{thm:gen} holds for all $d\geq 0$: a plane map of outer degree~$d$ admits a \ddm-orientation if and only if it has girth at least~$d$, and in this case it admits a unique suitable \ddm-orientation. 
Moreover, in the case $d=0$, the unique suitable $0/(-\!2)$-orientation is the geodesic biorientation.
\end{prop}

\begin{proof}
We only need to prove the case $d=0$ of this statement (since the case $d>0$ is proved in Section~\ref{sec:proof}). Let $M$ be a plane map of outer degree 0. Define a \emph{vertex-labelling} of $M$ as a labelling of its inner vertices by values in $\mathbb{Z}$ such that the difference of label between two adjacent vertices is at most $1$ (in absolute value) and the label of the outer vertex is $0$. To a vertex-labelling of $M$ we associate a weighted biorientation as follows: an edge $\{v,v'\}$ with $\mathrm{label}(v')=\mathrm{label}(v)$ is oriented $0$-way with weight $-1$ on each half-edge, while an edge $\{v,v'\}$ with $\mathrm{label}(v')=\mathrm{label}(v)+1$ is oriented $1$-way from $v$ to $v'$ with weight $-2$ on the outgoing half-edge and 0 on the ingoing half-edge. This mapping is easily seen to be a bijection between the vertex-labellings and the $0/(-\!2)$-orientations of $M$. Moreover, a $0/(-\!2)$-orientation is accessible if and only if each inner vertex has a neighbor of smaller label in the associated vertex-labelling. Furthermore the unique vertex-labelling such that each inner vertex has a neighbor of smaller label is the \emph{distance labelling} (where each vertex is labelled by its root-distance $D(v)$). Lastly the geodesic biorientation is minimal (it is even acyclic), hence suitable. Thus the geodesic biorientation is the unique suitable $0/(-\!2)$-orientation of $M$.
\end{proof}

We now consider the specialization of the master bijection to suitably $0/(-\!2)$-oriented maps. 
It is proved in \cite{BeFu10}
that the master bijection $\Phi$ as described in Definition~\ref{def:master-bijections} gives a bijection between plane maps of outer degree 0 endowed with a suitable weighted biorientation, and the weighted mobiles of excess 0 (the parameter correspondence is indicated in Figure~\ref{fig:parameter-correspondence}). 
We can now reformulate the definition of $d$-branching mobiles so as to include $d=0$: the definition is unchanged except that half-edges are required to have weight in $d\cup \{1,2,3\ldots\}$ if they are incident to white vertices and in $\{-2,-1,0\}\backslash \{d\}$ if they are incident to black vertices (thus the definition is unchanged for $d>0$). The above discussion (and the easy fact that 0-branching mobiles have excess 0) implies the following result.
\begin{prop}
Theorem~\ref{thm:girthd} holds for all $d\geq 0$, that is, plane maps of outer degree~$d$ and girth at least~$d$ are in bijection with $d$-branching mobiles.
\end{prop}

\begin{figure}
\begin{center}
\includegraphics[width=\linewidth]{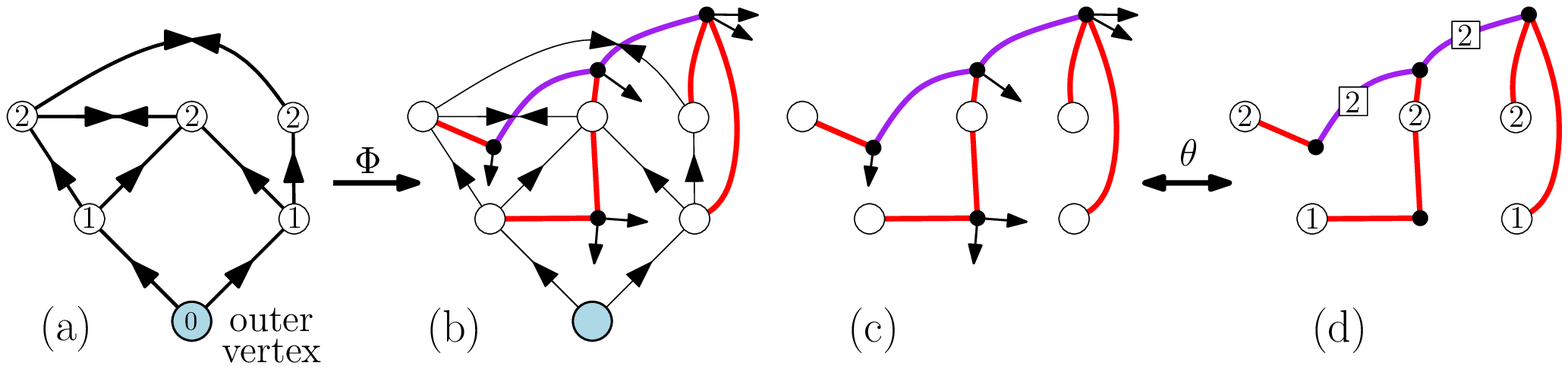}
\end{center}
\caption{
Bijection in the case $d=0$. (a) The suitable $0/(-2)$-orientation (i.e., geodesic biorientation). (b) The master bijection. (c) The 0-branching mobile (without the half-edges weights which are redundant). (d) The corresponding well-labelled mobile.}
\label{fig:bij_d=0}
\end{figure}

We now explain the relation between the case $d=0$ of our bijections and the bijection obtained by Bouttier, Di Francesco and Guitter~\cite{BDFG:mobiles}. 
A \emph{labelled mobile} is a mobile without buds or white-white edges, with a \emph{fake} white vertex added in the middle of each black-black edge, and having an integer label on each white vertex which is positive on non-fake white vertices and nonnegative on fake white vertices. 
For a corner $c$ incident to a black vertex, the \emph{jump} $\delta(c)$ is obtained from the labels $\ell,\ell'$ of the white vertices $v,v'$ preceding and following $c$
in clockwise order around the mobile by: $\delta(c)=\ell-\ell'$ if $v'$ is fake and $\delta(c)=\ell-\ell'+1$ otherwise. A \emph{well-labeled mobile} is a labelled mobile such that every jump is non-negative, and there is a non-fake white vertex of label 1 or a fake white vertex of label 0; an example is shown in Figure~\ref{fig:bij_d=0}(d). In~\cite{BDFG:mobiles} it was shown that plane maps of outer degree 0 are in bijection with well-labelled mobiles. The bijection can be described as follows: given a plane map $M$ of outer degree 0, one first endows $M$ with its geodesic biorientation (i.e., its suitable $0/(\!-\!2)$-orientation), and then draws the mobile in the same way as the master bijection $\Phi$, but forgets the buds and instead records the root-distance of each vertex and add a fake white vertex with label $\ell$ on each black-black edge of the mobile corresponding to a 0-way edge of $M$ between vertices both at root-distance~$\ell$.


It remains to explain the relation between well-labeled mobiles and 0-branching mobiles. Observe first that the weights are redundant for $0/(\!-\!2)$-orientations and $0$-branching mobiles. In particular, 0-branching mobiles identify with unweighted mobiles without white-white edges such that every black vertex has as many buds as white neighbors. Now, given a well-labelled mobile $L$, one obtains a 0-branching mobile $\theta(L)$ by adding $\delta(c)$ buds in each corner $c$ incident to a black vertex (and forgetting the labels and fake white vertices); see Figure~\ref{fig:bij_d=0}(c)--(d). The mapping $\theta$ is clearly a bijection. Moreover, if $L$ is the image of a plane map $M$ through the bijection described in~\cite{BDFG:mobiles}, then $\theta(L)$ is the image of $M$ through the master bijection $\Phi$. To summarize, well-labelled mobiles can be identified with $0$-branching mobiles and the bijection described in~\cite{BDFG:mobiles} coincides with the case $d=0$ of our bijection up to this identification.

\section{Proofs}\label{sec:proof}
In this section we prove Theorems~\ref{thm:bip_annular} and~\ref{thm:gen_annular} (which extend Theorems~\ref{thm:bip} and~\ref{thm:gen}) about \bbm-orientations and \ddm-orientations. In Subsection~\ref{sec:necessary} we prove that the conditions on girth are necessary to admit a \bbm-orientation or a \ddm-orientation. In Subsection~\ref{sec:proof_bip} we prove that for $b\geq 2$, any bipartite annular map $A$ in $\cA\rsb$ admits a unique suitable \bbm-orientation. In Subsection~\ref{sec:proof_bip} we prove that for $d\geq 2$, any annular map $A$ in $\cA\pqd$ admits a unique suitable \ddm-orientation. Lastly, in Subsection~\ref{sec:proofb=1} we treat the cases $b=1$ and $d=1$.

\subsection{Necessity of the conditions on cycle lengths}\label{sec:necessary}
\begin{lem}\label{lem:necd}
Let $d,p,q$ be positive integers with $p\leq q$. Let $M$ be an annular map of type $(p,q)$. If $M$ admits a \ddm-orientation, then $M$ is in $\cA\pqd$, that is, separating cycles have length at least $p$, and non-separating cycles have length at least~$d$.
\end{lem}
\begin{proof}
Let $M$ be an annular map of type $(p,q)$ admitting a \ddm-orientation. Let $C$ be a simple cycle of $M$, and let $\ell$ be its length. Let $v$, $e$, and $n$ be the numbers of vertices, edges, and faces strictly inside of $C$. Let $S$ be the sum of weights of the outgoing half-edges having a face inside $C$ on their right. Clearly Conditions (i) and (ii) of \ddm-orientations imply $d\,v\leq (d-2)e-S$. 
Suppose first that the cycle $C$ is non-separating. In this case, Condition (iii) of \ddm-orientations gives $d\,n=S+\sum_f\deg(f)$, where the sum is over the faces strictly inside $C$. Since $\sum_f\deg(f)=2e+\ell$, we get $S=d\,n-2e-\ell$. Thus, $d\,v\leq (d-2)e-d\,n+2e+\ell$, that is, $\ell\geq d(v-e+n)=d$, where the last equation is the Euler relation. This proves that non-separating cycles have length at least $d$. 
Suppose now that the cycle $C$ is separating. One still has $d\,v\leq (d-2)e-S$ and $\sum_f\deg(f)=2e+\ell$,
but Condition (iv) gives $d\,n+p-d=S+\sum_f\deg(f)$, hence $S=d\,n+p-d-2e-\ell$. Thus $d\,v\leq (d-2)e-d\,n+d-p+2e+\ell$, and $\ell\geq d(v-e+n)+p-d=p$. This proves that separating cycles have length at least $p$. 
\end{proof}

\begin{cor}\label{cor:necess}
Let $b$, $r$, $s$ be positive integers such that $r\leq s$.
Let $M$ be a bipartite annular map of type $(2r,2s)$.
If $M$ admits a \bbm-orientation, then $M$ is in $\cA\rsb$.
\end{cor}
\begin{proof}
If $M$ admits a \bbm-orientation, then doubling the weights of inner half-edges gives a \ddm-orientation for $d=2b$, so $M$ is in $\cA\rsb$ by Lemma~\ref{lem:necd}.
\end{proof}

\subsection{Existence and uniqueness of suitable \bbm-orientations for $b\geq 2$}\label{sec:proof_bip}
In this subsection we prove Theorem~\ref{thm:bip_annular} for $b\geq 2$. This is done in three steps which are represented in Figure~\ref{fig:compute_annular}. First we prove the existence of \bbm-orientations for annular $2b$-angulations in $\cA\rrb$ (Proposition~\ref{prop:existsForArb}). Then, for a bipartite map $M$ in $\cA\rsb$, we obtain the existence of certain orientations, called \emph{coherent regular orientations}, on a related map denoted $\MQ$ (Proposition~\ref{prop:exists_regular_ori}). Lastly, we use the coherent orientations of $\MQ$ in order to establish the existence and uniqueness of a suitable \bbm-orientation for $M$ (Proposition~\ref{prop:transfer-to-bbm}).

We start with some definitions and preliminary results. 
Let $M$ be a map, let $\alpha$ be a function from the vertex set to $\NN=\{0,1,\ldots\}$ and let $\beta$ be a function from the edge set to $\NN$. An $\alpha/\beta$\emph{-orientation} of $M$ is an \nb-biorientation such that any vertex $v$ has weight $\alpha(v)$, and any edge $e$ has weight $\beta(e)$. We now recall a criterion given in \cite{BeFu10} for the existence of an $\alpha/\beta$-orientation.

\begin{lem}\label{lem:exists-alpha}
Let $M$ be a map with vertex set $V$ and edge set $E$, let $\alpha$ be a function from $V$ to $\NN$, and let $\beta$ be a function from $E$ to $\NN$. The map $M$ admits an $\alpha/\beta$-orientation if and only if
\begin{enumerate}
\item[(i)] $\sum_{v\in V}\alpha(v)=\sum_{e\in E}\beta(e)$,
\item[(ii)] for each subset $S$ of vertices, $\sum_{v\in S}\alpha(v)\geq \sum_{e\in E_S}\beta(e)$ where $E_S$ is the set of edges with both ends in $S$.
\end{enumerate}
Moreover, $\alpha$-orientations are accessible from a vertex $u$ if and only if
\begin{enumerate}
\item[(iii)] for each subset $S\neq\emptyset$ of vertices not containing $u$, $\sum_{v\in S}\alpha(v)> \sum_{e\in E_S}\be(e)$.
\end{enumerate}
\end{lem}

For $b,r$ positive integers, we denote by $\cB_r^{(b)}$ the set of bipartite maps in $\cA\rrb$ such that every non-root face has degree $2b$ (in particular $\cB_b^{(b)}$ is the set of $2b$-angulations of girth $2b$).
\begin{prop}\label{prop:existsForArb}
For any positive integers $b,r$ with $b\ge 2$, every map $M$ in $\cB_r^{(b)}$ admits a \bbm-orientation, and the \bbm-orientations of $M$ are accessible from every outer vertex.
\end{prop}
\begin{proof}
Let $M$ be in $\cB_r^{(b)}$. Note that the \bbm-orientations of $M$ do not have half-edges with negative weight (because of Condition (iii) of \bbm-orientations). Hence, the \bbm-orientations of $M$ are exactly the $\alpha/\beta$-orientations, where $\alpha(v)=b$ for inner vertices, $\alpha(v)=1$ for outer vertices, $\beta(e)=(b-1)$ for inner edges, and $\beta(e)=1$ for outer edges. We will now use Lemma~\ref{lem:exists-alpha} to prove the existence of \bbm-orientations of $M$.

Let us check Condition~(i) first. Let $\vv$, $\ee$ and $\ff$ be the numbers of
vertices, edges, and faces of $M$. Two faces of $M$ have degree $2r$ and all the other faces have degree $2b$, hence $2\ee=2b(\ff-2)+4r$. Combining this with the Euler relation gives $b\vv=(b-1)\ee+2r$. This can be rewritten as $b(\vv-2r)+2r=(b-1)(\ee-2r)+2r$, so Condition~(i) holds.
Now we check Conditions~(ii) and~(iii). Note that is it enough to check these conditions on \emph{connected subsets} $S$ (subsets such that the graph induced by $S$ is connected). Indeed both quantities $\sum_{v\in S}\alpha(v)$ and $\sum_{e\in E_S}\beta(e)$ are additive over
non-adjacent connected subsets.
So we consider a subset $S$ of vertices of $M$ forming a connected submap, which we denote by $M_S$. Let $E_S$ and $F_S$ be respectively the sets of edges and faces of $M_S$, and let $\vv_S=|S|$, $\ee_S=|E_S|$ and $\ff_S=|F_S|$. We treat three cases.

 Assume first that $S$ contains all the outer vertices of $M$.
Since the separating girth is $2r$, the inner face of $M_S$ containing the inner root-face
of $M$ has degree at least $2r$. In addition the outer face of $M_S$ has degree $2r$, and all the other
faces have degree at least $2b$ (since the non-separating girth is at least $2b$).
Hence $2\ee_S\geq 2b(\ff_S-2)+4r$, which together with the Euler relation gives $b \vv_S\geq (b-1)\ee_S+2r$. By definition of \bbm-orientations, we have $\sum_{v\in S}\alpha(v)=b(\vv_S-2r)+2r$, and $\sum_{e\in E_S}\beta(e)=(b-1)(\ee_S-2r)+2r$. Hence
$\sum_{v\in S}\alpha(v)-\sum_{e\in E_S}\beta(e)=b \vv_S-(b-1)\ee_S-2r\geq 0$.

Assume now that $S$ contains at least one but not all outer vertices of $M$. Let $A$ be the set
of outer vertices not in $S$, and let $B$ be the set of outer edges not in $E_S$. Note that $|A|<|B|$.
 Let $S'=S\cup A$. The submap $M_{S'}$ contains all outer vertices (case already treated), hence
$\sum_{v\in S'}\alpha(v)-\sum_{e\in E_S'}\beta(e)\geq 0$. Moreover $\sum_{v\in S'}\alpha(v)=|A|+\sum_{v\in S}\alpha(v)$ and $\sum_{e\in E_{S'}}\beta(e)\geq |B|+\sum_{e\in E_S}\beta(e)$. Thus,
$$
\sum_{v\in S}\alpha(v)-\sum_{e\in E_S}\beta(e)\geq \Big(\sum_{v\in S'}\alpha(v)-\sum_{e\in E_{S'}}\beta(e)\Big)+|B|-|A|>0.
$$

Assume now that $S$ contains no outer vertex. In this case, $\sum_{v\in S}\alpha(v)=b\vv_S$ and $\sum_{e\in E_S}\beta(e)=(b-1)\ee_S$. Note that the contours of at most two faces of $M_S$ separate the two marked faces,
such faces have degree at least $2r$ (since the separating girth is at least $2r$) and
the other faces have degree at least $2b$ (since the non-separating girth is at least $2b$).
Hence $\ee_S\geq b(\ff_S-2)+2\ \!\mathrm{min}(r,b)$, which together
with the Euler relation gives $b\vv_S\geq (b-1)\ee_S+2\ \!\mathrm{min}(r,b)$, so $\sum_{v\in S}\alpha(v)>\sum_{e\in E_S}\beta(e)$.

Hence in all three cases, Condition~(ii) holds.
Note that the only case where the inequality $\sum_{v\in S}\alpha(v)\geq\sum_{e\in E_S}\beta(e)$ can be tight is if $S$ contains all the outer vertices of $M$. Hence Condition (iii) also holds.
\end{proof}

We now fix positive integers $b,r,s$ with $2\leq b$, $r\leq s$, and a bipartite map $M$ in $\cA\rsb$, and we prove the existence of a unique suitable \bbm-orientation of $M$ (thereby completing the proof of Theorem~\ref{thm:bip_annular} for $b\geq 2$). In order to prove this result, we consider some orientations on a related map denoted $\MQ$.
\fig{width=11cm}{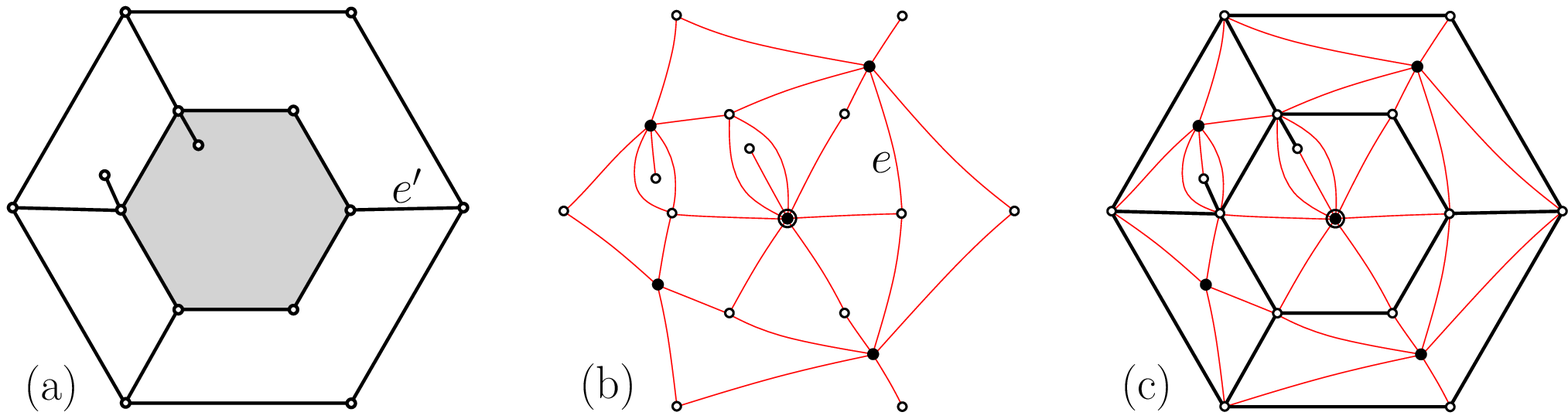}{An annular map $M$ (a), its inner-quadrangulation $Q$ (b), and the associated map $\MQ$ (c). The edge $e'$ is the $M$-edge of the edge $e$ of $Q$.}

We define the \emph{inner-quadrangulation} of $M$ to be the map $Q$ obtained from $M$ by inserting a vertex $v_f$, called \emph{face-vertex}, in each inner face $f$ of $M$, adding an edge from $v_f$ to each corner around $f$, and finally deleting the edges (but not the vertices) of $M$. An example is given in Figure~\ref{fig:M-Q-MQ}.
 The face-vertex in the inner root-face is called \emph{special}. Inner faces of $Q$ have degree 4 and correspond to the inner edges of $M$. We denote by $\MQ$ the map obtained by superimposing $M$ and~$Q$.

\begin{Def}\label{def:regular}
We call \emph{regular orientation of} $\MQ$, an admissible \nb-biorientation such that
\begin{enumerate}
\item[(i)] inner edges of $M$ have weight $b-1$, edges of $Q$ have weight~$1$,
\item[(ii)] inner vertices of $M$ have weight $b$,
\item[(iii)] any non-special face-vertex $v$ has a weight and degree satisfying $w(v)-\deg(v)/2=b$,
\item[(iv)] the special face-vertex $v$ has a weight and degree satisfying \mbox{$w(v)-\deg(v)/2=r$.}
\end{enumerate}
\end{Def}

We now prove the existence of a regular orientation of $\MQ$. Roughly speaking, we will use Proposition~\ref{prop:existsForArb},
 together with the fact that $M$ can be completed into a map in $\cB_r^{(b)}$, and the following result.
\begin{claim}\label{claim:indegree}
Let $B$ be a map in $\cB_r^{(b)}$ endowed with a \bbm-orientation. Let $C$ be a simple cycle of length $2c$ having only inner vertices. Then, the sum $S$ of weights of the ingoing half-edges incident to $C$ but strictly outside of the region enclosed by $C$ is $r+c$ if $C$ encloses the inner root-face and $b+c$ otherwise.
\end{claim}
\begin{proof}
By the Euler relation and the incidence relation between faces and edges, the numbers $v$ and $e$ of vertices and edges inside of $C$ (including vertices and edges on $C$) satisfy $(b-1)e+b+c=b\,v$ if $C$ encloses the inner root-face (which has degree $2r$) and 
 $(b-1)e+r+c=b\,v$ otherwise. Moreover $b\,v=S+(b-1)e$ since both sides equal the sum of weights of ingoing half-edge incident to vertices inside $C$ (including vertices on $C$). This proves the claim. 
\end{proof}

\begin{prop}\label{prop:exists_regular_ori}
There exists a regular orientation of $\MQ$ which is accessible from every outer vertex of~$M$.
\end{prop}
\fig{width=\linewidth}{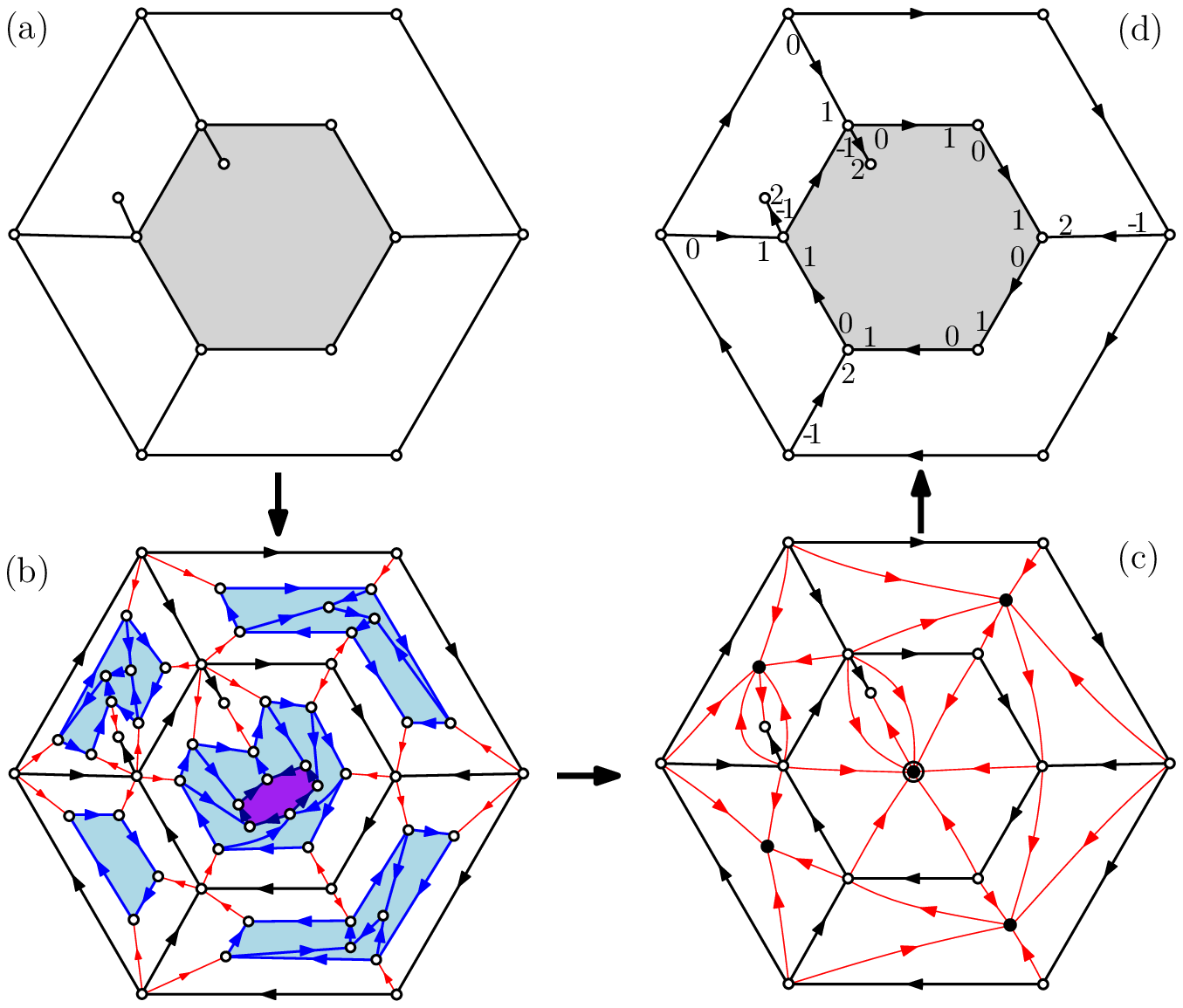}{Process for constructing a \bbm-orientation of a bipartite map
$M\in\cA\rsb$ ($b=2$, $r=3$, $s=4$ here). In (b) $M$ is completed into a map $B\in\cB_r^{(b)}$, and $B$ is endowed with a \bbm-orientation. In (c) the \bbm-orientation orientation of $B$ is contracted into a regular orientation $X$ of $\MQ$. In (d) the map $M$ gets the \bbm-orientation $\si(X)$.}

\begin{proof}
The proof is illustrated in Figure~\ref{fig:compute_annular}.
The first step is to complete $M$ into a map $B\in\cB_r^{(b)}$, by adding vertices and edges
inside each inner face of $M$. We make use of the following basic facts (valid for $b\geq 2$):
\begin{itemize}
\item For each integer $j\leq b$, there exists a plane map $L_j$ of girth $2b$, whose outer face is a simple cycle $C$ of length $2j$, whose inner faces have degree $2b$, and such that for any pair $u,v$ of vertices on $C$
the distance between $u$ and $v$ on $C$ is the same as the distance between $u$ and $v$ in $L_j$.
\item For $r\leq s$ there exists an annular map $L_{s,r}$ where the contour $C$ of the outer face (of degree $2s$) and the contour of the inner root-face (of degree $2r$) are simple cycles, with all non-root faces of degree $2b$, with separating girth $2r$, non-separating girth $2b$, and with the following property: ``For any path $P\subset C$, and any path $P'$ having the same endpoints, if the cycle $P\cup P'$ is not separating, then then the length of $P'$ is greater or equal to the length of $P$."
\end{itemize}
 Now, for each inner face $f$ of $M$, we consider the sequence of the corners $c_1,\ldots,c_{\deg(f)}$ in clockwise order around $f$ (note that the incident vertices of these corners might not be all distinct). We ``throw" a simple path $P_i$ of length $(b-1)$, called a \emph{transition-path}, from each corner $c_i$
toward the interior of $f$. Denote by $v_i$ the vertex at the free extremity of each path $P_i$. We then connect the vertices $v_i$ along a simple cycle $C_f$ (of length $\deg(f)$), the order of the vertices along the cycle corresponding to the order of the
corners around $f$. 
Then we patch a copy of the plane map $L_{\deg(f)/2}$ inside $C_f$ if $f$ is not the inner root-face,
and we patch a copy of $L_{r,s}$ in the cycle $C_f$ in the inner root-face.
We obtain
a bipartite map $B\in\cB_r^{(b)}$ (it is easily checked that, thanks to the distance properties of the patched maps, 
the non-separating girth stays greater or equal to $2b$, and 
 the separating girth stays the same). 
By Proposition~\ref{prop:existsForArb}, $B$ admits a \bbm-orientation $O$.
Let $P$ be a transition path, and let $w,w'$ be the weights of the half-edges $h,h'$ at the extremities of $P$. We claim that $\{w,w'\}=\{0,1\}$. Indeed, since the weight of each edge of $P$ is $b-1$, and since the weights of the two half-edges incident to each inner vertex of $P$ add up to $b$, we have $w+w'=1$ (and the two weights are non-negative).
We then perform the following operations to obtain an orientation of $\MQ$:
\begin{enumerate}
\item we shrink each transition-path $P$ into a single 1-way edge of weight 1 by only keeping the extremal half-edges $h,h'$ of weight $w,w'$,
\item we contract the cycle inserted inside each inner face (and the map $L_j$ contained therein) into a single
vertex, which becomes a face-vertex of $Q$.
\end{enumerate}
We claim that the obtained orientation $X$ of $\MQ$ is a regular orientation. Indeed, the weights of the vertices and edges of $M$ are the same as in the \bbm orientation $O$, that is, inner (resp. outer) vertices of $M$ have weight $b$ (resp. $1$)
 and inner (resp. outer) edges of $M$ have weight $b-1$ (resp. $1$). Moreover, Claim~\ref{claim:indegree} implies that any non-special face-vertex $v$ has weight $w(v)=\deg(v)/2+b$, and the special face-vertex $v$ has weight $w(v)=\deg(v)/2+r$.
Thus $X$ is a regular orientation of $\MQ$, and it only remains to show that it is accessible from every outer vertex. Now, the orientation $O$ is accessible from the outer vertices, and moreover the operations for going from $O$ to $X$ can only increase the accessibility between vertices. Thus, $X$ is accessible from the outer vertices.
\end{proof}

We have established the existence of a regular orientation of $\MQ$. We will now complete the proof of the existence and uniqueness of a suitable \bbm-orientation of $M$. We first define a mapping between certain regular orientations of $\MQ$ and the \bbm-orientations of $M$. 

For an edge $e$ of $Q$, with $v$ the endpoint of $e$ in $M$, we call $M$\emph{-edge} of $e$ the edge of $M$ following $e$ in clockwise order around $v$. An example is given in Figure~\ref{fig:M-Q-MQ}.
A regular orientation of $\MQ$ is said to be \emph{coherent} if for each edge $e$ of $Q$ directed toward its endpoint $v$ in $M$, the corresponding $M$-edge $e'$ is a 1-way edge oriented toward $v$ (observe that in this case $v$ is an inner vertex and $e'$ is an inner edge of $M$).
 Given a coherent regular orientation $X$ of $\MQ$, we denote by $\si(X)$ the \zb-biorientation of $M$ obtained by
\begin{itemize}
\item keeping the biorientation of every edge $e$ of $M$ in $X$
\item keeping the weights on the half-edges, except if $e$ is the $M$-edge of an edge of $Q$ directed toward its endpoint in $M$, in which case the weights $0$ and $b-1$ are replaced by $-1$ and $b$ respectively.
\end{itemize}
It is now sufficient to show the following property of the mapping $\sigma$.

\begin{prop}\label{prop:transfer-to-bbm}
The mapping $\sigma$ is a bijection between the coherent regular orientations of $\MQ$ and the \bbm-orientations of $M$. Moreover, there exists a unique coherent regular orientation $O$ of $\MQ$ such that its image $\sigma(O)$ is suitable.
\end{prop}

The first part of Proposition~\ref{prop:transfer-to-bbm} is easy to establish.
\begin{lem}\label{lem:regular-to-bbm}
The mapping $\sigma$ is a bijection between the coherent regular orientations of $\MQ$ and the \bbm-orientations of $M$.
\end{lem}
\begin{proof}
For any coherent regular orientation $X$ of $\MQ$, the image $\si(X)$ is a \bbm-orientation of $M$ since Conditions (i), (ii), (iii), and (iv) of regular orientations correspond respectively to Conditions (i), (ii), (iii), and (iv) of \bbm-orientations. Moreover, the mapping $\si$ is easily seen to be surjective and invertible.
\end{proof}

In order to establish the second part of Proposition~\ref{prop:transfer-to-bbm}, we need to examine the properties of the regular orientations of $\MQ$.
\begin{lem} \label{lem:minimal-regular} 
The regular orientations of $\MQ$ are all accessible from every outer vertex. Moreover there exists a unique minimal regular orientation of $\MQ$, and this orientation is coherent.
\end{lem}
We use the following general result proved in \cite{BeFu10}.
\begin{lem}\label{lem:unique_minimal}
Let $G=(V,E)$ be a plane map. Let $\alpha$ be a function from $V$ to
$\mathbb{N}$ and and $\beta$ a function from $E$ to $\mathbb{N}$. If $G$ has an $\alpha/\beta$-orientation, then $G$ has a unique minimal $\al/\beta$-orientation.
\end{lem}

\begin{proof}[Proof of Lemma~\ref{lem:minimal-regular}]
Observe that there exist functions $\alpha$, $\beta$ depending on $M$ such that the regular orientations of $\MQ$ are exactly the $\alpha/\beta$-orientations of $\MQ$. Now, Proposition~\ref{prop:exists_regular_ori} asserts the existence of a regular orientation of $\MQ$ which is accessible from every outer vertex. Since the accessibility of the $\alpha/\beta$-orientations only depends on $\alpha$ and $\beta$ (see Condition (iii) of Lemma~\ref{lem:exists-alpha}), this implies that all the regular orientations are accessible from the outer vertices of $M$. Moreover Lemma~\ref{lem:unique_minimal} ensures that $\MQ$ has a unique minimal regular orientation $X$. It only remains to prove that $X$ is coherent. 

Let $e=\{a,v\}$ be an edge of $Q$ oriented toward its endpoint $v$ in $M$. Let $\eps=\{u,v\}$ be the $M$-edge of $e$, and let $f$ be the face of~$M$ containing~$e$. We want to show that $\eps$ is oriented 1-way toward $v$. Assume this is not the case, that is, the weight $i$ of the half-edge incident to $u$ is positive.
Let $e'=\{a,u\}$ be the edge preceding $e$ in clockwise order around $a$.
Since $\eps$ can be traversed from $v$ to $u$, the edge $e'$ must be oriented away from $a$ (otherwise the triangle $\{a,v,u\}$ would form a counterclockwise circuit, in contradiction with the minimality of $X$).
Let $\eps'$ be the $M$-edge of $e'$. Since the vertex $u$ has total weight $b$, with contribution $i>0$ by the edge $\eps$
and contribution $1$ by the edge $e'$, we conclude that the weight of $\eps'$ at $u$ is at most $b-2$,
hence $\eps'$ is not 1-way toward $v$. By the same arguments, the edge $e''$
preceding $e'$ in clockwise order around $a$ is also oriented away from $a$. Continuing in
this way around the face $f$ we reach the contradiction that all edges incident to $a$ are oriented away from $a$.
\end{proof}

The next two lemmas complete the proof of Proposition~\ref{prop:transfer-to-bbm} by showing that the image of a coherent regular orientation $X$ by the mapping $\si$ is suitable if and only if $X$ is the minimal regular orientation of $\MQ$.

\begin{lem}\label{lem:minmintau}
The minimal regular orientation of $\MQ$ is mapped by $\sigma$ to a suitable \bbm-orientation of $M$.
\end{lem}
\begin{proof}
Let $X$ be the minimal regular orientation of $\MQ$, and let $Y=\sigma(X)$. We want to prove that $Y$ is minimal and accessible from every outer vertex. The minimality of $Y$ is obvious since $Y$ is a suborientation (forgetting
the weights) of $X$.
We now consider an outer vertex $v_0$ and prove that $Y$ is accessible from $v_0$.
Let $k$ be the number of inner faces of $M$, and let $b_1,\ldots,b_k$ be the face-vertices of $\MQ$.
Let $H_0$ be the underlying biorientation of $X$ (forgetting the weights),
and for $i\in \{1,\ldots,k\}$ let $H_i$ be the biorientation obtained from $H_0$ by deleting the face-vertices $b_1,\ldots,b_i$ and their incident edges. Note that $H_k$ is the underlying biorientation of $Y$.
Recall that $H_0$ is accessible from $v_0$. We will now show that $H_i$ is accessible from $v_0$ by induction on $i$.
We assume that $H_{i-1}$ is accessible from $v_0$ and suppose for contradiction that a vertex $w$ is not accessible from $v_0$ in $H_i$. In this case, each directed path $P$ from $v_0$ to $w$ in $H_{i-1}$ goes through $b_i$.
Let $P$ be such a path, and let $e_0$ and $e_1$ be the edges arriving at and leaving $b_i$ along $P$.
We define the \emph{left-degree} of $P$ to be the number of edges of $Q$ between $e_0$ and $e_1$ in clockwise order around $b_i$. 
We choose $P$ so as to minimize the left-degree. Call $P_0$ the portion of $P$ before $b_i$, and $P_1$ the portion of $P$ after $b_i$. Let $u$ be the origin of $e_0$, let $v$ be the end of $e_1$ and let $\eps$ be the $M$-edge of $e_1$; see Figure~\ref{fig:argument}.
Since the regular orientation $X$ is coherent, the edge $\eps$ is oriented 1-way toward $v$.
Call $v'$ the origin of $\eps$. Note that $v'\neq u$ (if $v'=u$ one could pass by $\eps$, thus avoiding $b_i$, to go from $v_0$ to $w$) and that the edge $e_1'=\{b_i,v'\}$ preceding $e_1$ in clockwise order around $b_i$ must be directed from $v'$ to $b_i$ (otherwise one could replace in $P$ the portion $u\to b_i\to v$ by $u\to b_i\to v'\to v$, yielding a path with smaller left-degree, a contradiction).
Since $H_{i-1}$ is accessible from $v_0$, there exists a directed path $P'$ in $H_{i-1}$ from $v_0$ to $v'$. We can choose $P'$ in such a way that it shares an initial portion with $P$ but does not meet again $P$ once it leaves it.
Note that $v'$ is not accessible from $v_0$ in $H_i$ (if it was, so would be $v$, hence so would be $w$), so $P'$ has to pass by $b_i$, so the portion of $P'$ before $b_i$ equals $P_0$.
Let $e'$ be the edge taken by $P'$ when it leaves $b_i$ and let $P_1'$
 be the portion of $P'$ after $e'$.
 Note that $e'$ can not be strictly between $e_0$ and $e_1$ in clockwise order around $b_i$
 (otherwise by a similar argument as above, it would be possible to produce a path
 from $v_0$ to $w$ with smaller left-degree than in $P$).
 Since $P_1'$ can not meet $P_0$ again, it has to form a counterclockwise circuit together with the two edges
 $e_1'$ and $e'$, see Figure~\ref{fig:argument}. We reach a contradiction.
 This concludes the proof that $H_i$ is accessible from $v_0$. By induction on $i$, the biorientation $H_k$ underlying $Y$ is accessible from $v_0$. Thus $Y$ is accessible from every outer vertex. Hence $Y$ is suitable.
\end{proof}

\fig{width=8cm}{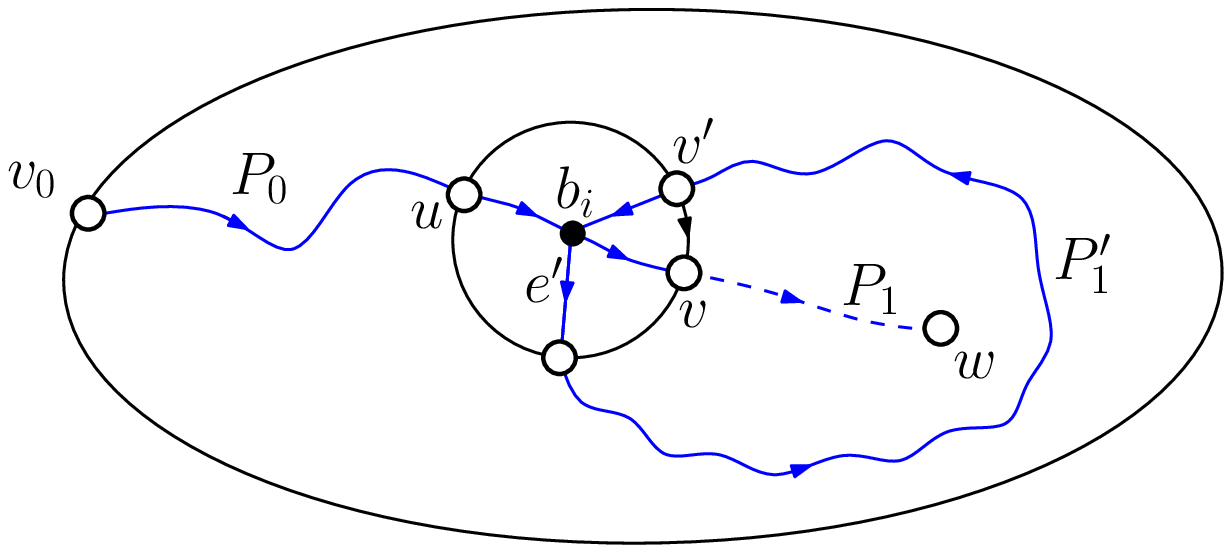}{The situation in the proof of Lemma~\ref{lem:minmintau}.}

\begin{lem}\label{lem:bijtau}
If $X$ is a coherent regular orientation of $\MQ$ which is not minimal, then its image by $\si$ is not suitable.
\end{lem}
\begin{proof}
Suppose for contradiction that the \bbm-orientation $Y=\si(X)$ is suitable.
Since $X$ is not minimal, it has a simple counterclockwise circuit $C$. By choosing $C$ to enclose no other counterclockwise circuit, we can assume that $C$ has no \emph{chordal path} (a chordal path is a directed path strictly inside of $C$ connecting two vertices of $C$). Since $Y$ has no counterclockwise circuit, $C$ must contain at least one edge $e$ of $Q$ oriented toward its endpoint $v$ in $M$. Since $X$ is coherent, the $M$-edge $\eps$ of $e$ (note that $\eps$ is strictly inside $C$) is oriented 1-way toward $v$.
 Let $v_0$ be an outer vertex.
Since $X$ is accessible from $v_0$ (by Lemma~\ref{lem:minimal-regular}), there exists an oriented path in $X$ from $v_0$ to $v$ ending at the edge $\eps$. The portion $P$ of the path inside $C$ is a chordal path for $C$, yielding a contradiction.
\end{proof}

We have proved Proposition~\ref{prop:transfer-to-bbm} through Lemmas~\ref{lem:regular-to-bbm},~\ref{lem:minimal-regular},~\ref{lem:minmintau} and~\ref{lem:bijtau}. This establishes that any bipartite map $M\in\cA\rsb$ has a unique suitable \bbm-orientation. The necessity of being in $\cA\rsb$ was established in Corollary~\ref{cor:necess}. This concludes the proof of Theorem~\ref{thm:bip_annular} for $b\geq 2$.

\subsection{Existence and uniqueness of a \ddm-orientation for $d\geq 2$} \label{sec:dgeq2}
In this subsection, we fix positive integers $d,p,q$ such that $p\leq q$, and consider
a map $M\in\cA\pqd$.
We will prove that if $d\geq 2$, then $M$ admits a unique suitable \ddm-orientation (thereby completing the proof of Theorem~\ref{thm:gen_annular} for $d\geq 2$). This will be done by a reduction to the bipartite case as illustrated in Figure~\ref{fig:reduc_gen}.

We denote by $M'$ the map obtained from $M$ by inserting a vertex $v_e$, called an \emph{edge-vertex}, in the middle of each edge $e$ of $M$. Clearly $M'$ is bipartite and is in $\cA\Bpqd$ since
cycle lengths are doubled. Given a \zb-biorientation of $M'$, the \emph{induced orientation} on $M$ is the \zb-biorientation of $M$ obtained by contracting the edge-vertices and their two incident half-edges (the two other half-edges get glued together).
\begin{claim}\label{claim:d-1}
Let $X$ be a $d/(d-1)$-orientation of $M'$, and let $Y$ be the induced orientation on $M$. Then for any inner edge of $M$ the weights $i,j$ on the half-edges satisfy either $i<d$, $j<d$ and $i+j=d-2$, or $\{i,j\}=\{-1,d\}$ in which case the edge is called \emph{special}.
\end{claim}
\begin{proof}
This is an easy consequence of the fact that the edge-vertex $v_e$ has weight~$d$, and the two edges of $M'$ incident to $v_e$ have weight $d-1$.
\end{proof}
Let $X$ be a $d/(d-1)$-orientation of $M'$, and let $Y$ be the induced orientation of $M$. We denote by $\tau(X)$ the admissible \zb-biorientation of $M$ obtained from $Y$ by replacing the weights $-1$ on special inner edges by $-2$.

\begin{figure}
\begin{center}
\includegraphics[width=10cm]{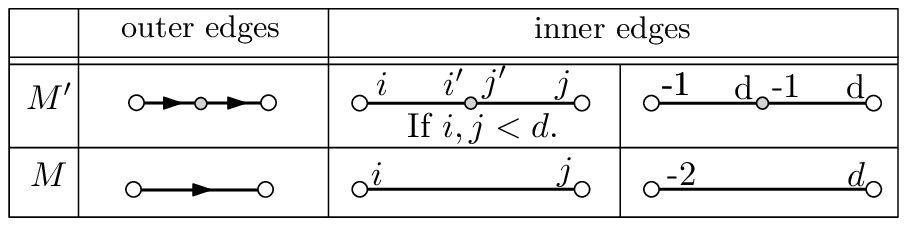}\\[1cm]
\includegraphics[width=\linewidth]{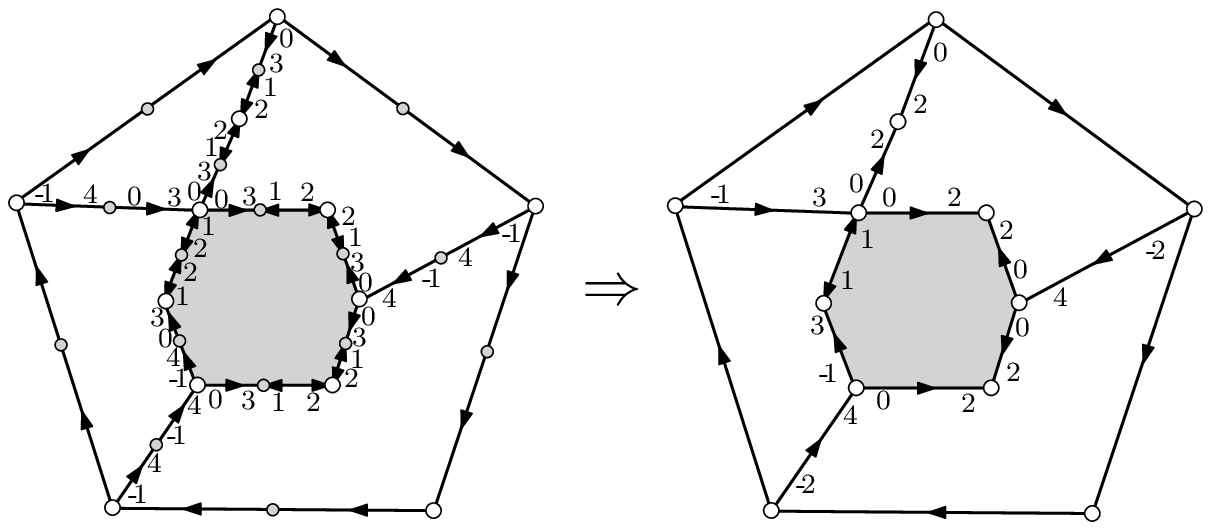}
\end{center}
\caption{Top: The mapping $\tau$ between $d/(d-1)$-orientations of the bipartite map $M'$ (obtained from $M$ by inserting a vertex in the middle of each edge) and \ddm-orientations of $M$.
 Bottom: example in the case $d=4$, with type $(p,q)=(5,6)$.}
\label{fig:reduc_gen}
\end{figure}

\begin{lem}\label{lem:taubij}
The mapping $\tau$ is a bijection between the $d/(d-1)$-orientations of $M'$ and the \ddm-orientations of $M$. Moreover, a $d/(d-1)$-orientation $X$ is suitable if and only if $\tau(X)$ is suitable.
\end{lem}
\begin{proof}
Let $X$ be a $d/(d-1)$-orientation. Claim~\ref{claim:d-1} implies that $\tau(X)$ satisfies Conditions (i) and (ii) of \ddm-orientations. Moreover Condition (iii) follows from the fact that the weights of faces are preserved by $\tau$. Thus $\tau(X)$ is a \ddm-orientation.
Clearly $\tau$ is a bijection; indeed a $d/(d-1)$-orientation of $M'$ is completely determined by its contraction (the weight of every edge is fixed), and the rule (replacing each edge $(-1,d)$
by an edge $(-2,d)$) to go from contracted $d/(d-1)$-orientations to $d/(d-2)$-orientations of $M$ 
is invertible. 

It remains to prove the second assertion. Let $X$ be any $d/(d-1)$-orientation of $M'$. For any inner edge $e=\{u,v\}$ of $M$, it is easy to see that the edge $e$ can be traversed from $u$ to $v$ in $\tau(X)$ (that is, $e$ is 2-way or 1-way toward $v$) if and only if the path of $M'$ made of the edges $e_1=\{u,v_e\}$ and $e_2=\{v_e,v\}$ can be traversed from $u$ to $v$ in $X$. Therefore, the orientation $\tau(X)$ is minimal if and only if $X$ is minimal. Moreover, for any edge $e=\{u,v\}$ of $M$, it is either possible to traverse $e_1=\{u,v_e\}$ from $u$ to $v_e$, or to traverse $e_2=\{v_e,v\}$ from $v$ to $v_e$ (since the weight of $v_e$ is positive). Hence, the orientation $\tau(X)$ is accessible from a vertex $v_0$ if and only if $X$ is accessible from $v_0$. Thus a $d/(d-1)$-orientation $X$ is suitable if and only if $\tau(X)$ is suitable.
\end{proof}

We now suppose that $d\geq 2$ and prove that $M$ admits a unique suitable \ddm-orientation. Since $d\geq 2$, it has been proved in subsection~\ref{sec:proof_bip} that $M'$ admits a unique suitable $d/(d-1)$-orientation. Therefore, Lemma~\ref{lem:taubij} implies that the map $M$ admits a unique suitable \ddm-orientation. This concludes the proof that for all positive integers $d,p,q$ with $2\leq d$ and $p\leq q$, every map in $\cA\pqd$ admits a unique suitable \ddm-orientation. The necessity of being in $\cA\pqd$ has been proved in Lemma~\ref{lem:necd}. This concludes the proof of Theorem~\ref{thm:gen_annular} in the case $d\geq 2$.

 \subsection{Existence and uniqueness for $b=1$ and $d=1$} \label{sec:proofb=1}

We first prove the case $b=1$ of Theorem~\ref{thm:bip_annular}. Let $M$ be a bipartite map in $\cA_2^{(2r,2s)}$. The case $d=2$ of Theorem~\ref{thm:gen_annular} (which has already been proved) implies that $M$ admits a unique suitable $2/0$-orientation $O$. Now, in order to prove that $M$ has a unique suitable $1/0$-orientation, it suffices to show that every inner half-edge of $O$ has even weight. Indeed, in this case one obtains a suitable $1/0$-orientation by dividing the inner weights of $O$ by two (and it is unique because any suitable $1/0$-orientation gives a suitable $2/0$-orientation by doubling the weights). In order to prove that the inner weights of $O$ are even, we consider the $2$-branching mobile $T$ of type $(2r,2s)$ associated to the map $M$ endowed with $O$ (we are using the case $d=2$ of Theorem~\ref{thm:girthd_annular} which has already been proved). We say that an edge of $T$ is odd if one of the half-edges has odd weight; in this case both half-edges have in fact odd weights since the weight of an edge is 0. It is easy to see that every vertex of $T$ has even weight (since the black vertices have even degree), so is incident to an even number of odd edges. This implies that the set of odd edges of $T$ is empty (since any non-empty forest has at least one leaf). The weight of every half-edge of $T$ is even, hence the weight of every inner half-edge of $O$ is even. This completes the proof of the case $b=1$ of Theorem~\ref{thm:bip_annular}.

We now establish the case $d=1$ of Theorem~\ref{thm:gen_annular} by using the same strategy as in Subsection~\ref{sec:dgeq2}. We want to prove that, for positive integers $p,q$ with $p\leq q$, 
a map $M\in\cA_{1}^{(p,q)}$ admits a unique suitable $1/(\!-\!1)$-orientation. We consider the bipartite map $M'$ obtained by inserting a vertex at the middle of each edge. As we have just proved (case $b=1$), the bipartite map $M'$ admits a unique suitable $1/0$-orientation. Therefore, Lemma~\ref{lem:taubij} implies that the map $M$ admits a unique suitable $1/(\!-\!1)$-orientation. This concludes the proof of the case $d=1$ of Theorem~\ref{thm:gen_annular}.

\bibliographystyle{plain}
\bibliography{biblio_gir}
\label{sec:biblio}
\end{document}